\documentclass[11pt,a4paper]{amsart}

\title[A calculus for finite parts and residues]{A calculus for finite parts and residues of some divergent complex geometric integrals}

\author{Ludvig Svensson}
\keywords{tbd}
\address{Ludvig Svensson, Department of Mathematical Sciences, University of Gothenburg and 
Chalmers University of Technology, SE-412 96 G\"{o}teborg, Sweden}
\email{ludsven@chalmers.se}

\date{\today}

\usepackage[english]{babel}
\usepackage{amsmath,calligra}
\usepackage{amsthm,amssymb,latexsym}
\usepackage[all,cmtip]{xy}
\usepackage{url,ae}
\usepackage{t1enc}
\usepackage{bm}
\usepackage{mathrsfs}
\usepackage{graphicx}
\usepackage{geometry}
\usepackage{enumitem}
\usepackage{tabularx}
\usepackage{mathtools}
\usepackage[nameinlink]{cleveref}
\usepackage{xcolor}  

\usepackage{microtype}

\crefname{figure}{Figure}{Figures}
\crefname{equation}{}{}
\crefname{table}{Table}{Tables}
\crefname{section}{§}{§§}
\crefname{appendix}{Appendix}{Appendices}

\newtheorem{proposition}{Proposition}[section]
\newtheorem{theorem}[proposition]{Theorem}
\newtheorem{lemma}[proposition]{Lemma}

\newtheorem{conjecture}[proposition]{Conjecture}
\theoremstyle{definition}
\newtheorem{definition}[proposition]{Definition}
\newtheorem{example}[proposition]{Example}

\numberwithin{equation}{section}

\crefname{corollary}{corollary}{corollaries}
\crefname{proposition}{proposition}{propositions}
\crefname{theorem}{theorem}{theorems}
\crefname{lemma}{lemma}{lemmas}
\crefname{definition}{definition}{definition}

\newcommand{\C}{\mathbb{C}}
\newcommand{\Q}{\mathbb{Q}}
\renewcommand{\d}{\mathrm{d}}

\def\newop#1{\expandafter\def\csname #1\endcsname{\mathop{\rm #1}\nolimits}}

\begin{document}
\nocite{*}
\bibliographystyle{plain}

\begin{abstract}
    We consider divergent integrals $\int_X \omega$ of certain forms $\omega$ on a reduced pure-dimensional complex space $X$. The forms $\omega$ are singular along a subvariety defined by the zero set of a holomorphic section $s$ of some holomorphic vector bundle $E$. Equipping $E$ with a smooth Hermitian metric allows us to define a finite part $\mathrm{fp}\,\int_X \omega$ of the divergent integral as the action of a certain current extension of $\omega$. We introduce a current calculus to compute finite parts for a special class of $\omega$. Our main result is a formula that decomposes the finite part of such an $\omega$ into sums of products of explicit currents. Lastly, we show that, in principle, it is possible to reduce the computation of $\mathrm{fp}\,\int_X \omega$ for a general $\omega$ to this class.
\end{abstract}

\maketitle
\thispagestyle{empty}

\section{Introduction}

Let $X$ be a reduced complex analytic space of pure dimension $n$ and let $V \subset X$ be an analytic subvariety. Suppose that $V$ is given by the vanishing locus of some holomorphic section $s \colon X \rightarrow E$ of some holomorphic vector bundle $E \rightarrow X$. Given a (smooth) Hermitian metric $\|\cdot\|$ on $E$, let $\mathcal{A}_{s,\|\cdot\|}(X)$ be the space of smooth differential forms $\omega$ on $X \setminus V$ such that, for each compact subset $K \subseteq X$, there exists some integer $N$ such that $\|s\|^{2N}\omega$ extends to a smooth form across $V \cap K$. In fact, we can let $N\in\mathbb{Q}_+$. Let $\mathcal{A}_s(X)$ be the union of all such $\mathcal{A}_{s,\|\cdot\|}(X)$. We call $\mathcal{A}_s(X)$ the space of \textit{quasi-meromorphic} forms on $X$ with singularities defined by the ideal $(s)$.

Any quasi-meromorphic form $\omega \in \mathcal{A}^{p,q}_{s}(X)$ defines a current on $X\setminus V$ of bidegree $(p,q)$ in a natural way (since $\omega$ is a smooth form on $X\setminus V$). It is often desirable to extend $\omega$ as a current to all of $X$. One way to do this is to consider the family of forms $\lambda \mapsto \|s\|^{2\lambda} \omega$ parameterized by $\lambda \in \C$, given some choice of Hermitian metric $\|\cdot\|$ on $E$. For $\mathfrak{Re}\,\lambda$ sufficiently large, $\|s\|^{2\lambda} \omega$ extends to a locally integrable form across $V$ and thus defines a current on all of $X$. Moreover, for any test form $\xi$, the integral
\begin{equation}
    \label{eq:gamma}
    \langle \|s\|^{2\lambda} \omega,\xi\rangle = \int_X \|s\|^{2\lambda} \omega \wedge \xi,
\end{equation}
which is known as an \textit{Archimedean local zeta function}, see \cite{I1}, is holomorphic as a function of $\lambda$ for $\mathfrak{Re}\,\lambda \gg 0$ and admits a meromorphic continuation to all of $\C$, with poles in a discrete subset of $\Q$. This classical result, originally due to Bernstein--Gel'fand and Atiyah, see \cite{BG} and \cite{Atiyah}, respectively, was used in, e.g., \cite{S1} to study the current $\|s\|^{2\lambda}\omega$ in a neighborhood of $\lambda = 0$. There is some $\kappa \leq n$ such that, in a neighborhood of $\lambda = 0$,
\begin{equation}
    \label{eq:laurentexpansion}
    \langle \|s\|^{2\lambda} \omega, \xi \rangle = \sum\limits_{j=-\kappa}^\infty \lambda^j \langle \mu_{j}^{\|s\|}(\omega), \xi \rangle,
\end{equation}
where $\mu_{j}^{\|s\|}(\omega)$ are currents on $X$ and $\mu_0^{\|s\|}(\omega)$ is a current extension of $\omega$ across $V$.

When $\omega \in \mathcal{A}_s(X)$ is of top degree and either $X$ is compact, or, more generally, $\overline{\mathrm{supp}\,\omega}$ is compact, we can use $\mu_0^{\|s\|}(\omega)$ to define \textit{a finite part} $\mathrm{fp}\int_X \omega$ of the divergent integral $\int_X \omega$ as
\begin{equation}
    \label{eq:finitepartdef}
    \mathrm{fp}\int_X \omega = \langle \mu_0^{\|s\|}(\omega),1\rangle.
\end{equation}
Improper integrals of the type $\int_X \omega$ appear, for instance, in superstring perturbation theory. In that setting $X$ is a moduli space of super Riemann surfaces with marked points, $V$ is the boundary divisor of some suitable compactification of $X$ and $\omega = \alpha \wedge \bar{\beta}$, where $\alpha$ and $\beta$ are meromorphic $(n,0)$-forms with polar set contained in $V$, see, e.g., \cite{W1}. 

\medskip

The main result of this paper is a description of the currents $\mu_j^{\|s\|}(\omega)$ for certain quasi-meromorphic forms $\omega$ which are built up by simpler elements. By simple element we mean the following: Assume that $s$ is a holomorphic section of a Hermitian line bundle $(L,\|\cdot\|)$ and let
\begin{equation}
    \label{eq:omegasimple}
    \omega = \bar{\partial}\log\|s\|^2 \wedge \partial \log\|s\|^2.
\end{equation}
Then $\omega\in\mathcal{A}_{s,\|\cdot\|}^{1,1}(X)$. We call $\omega$ an \textit{elementary} quasi-meromorphic form associated to the pair $(s,\|\cdot\|)$. We obtain an explicit formula for $\mu_{\ell}^{\|s\|}(\omega)$, for each $\ell$, see \Cref{lem:1} below. In particular, it turns out that $\mu_\ell^{\|s\|}(\omega) = 0$ for each $\ell < -1$, and
\begin{equation}
    \label{eq:mu-1}
    \mu_{-1}^{\|s\|}(\omega) = 2\pi i [\mathrm{div}(s)],
\end{equation}
where $[\mathrm{div}(s)]$ is the Lelong current associated to the divisor $\mathrm{div}(s)$.\footnote{If $X$ is not normal, then $\mathrm{div}(s) = p_*\mathrm{div}(p^*s)$, where $p \colon \widetilde{X} \rightarrow X$ is the normalization.}

\medskip

By a quasi-meromorphic form $\omega$ built up by simpler elements, we mean
\begin{equation}
    \label{eq:omegaproduct}
    \omega = \omega_1 \wedge \cdots \wedge \omega_\kappa,
\end{equation}
where $\omega_j$ is an elementary quasi-meromorphic form associated to a holomorphic section $s_j \colon X \rightarrow L_j$ of a holomorphic line bundle $L_j \rightarrow X$ equipped with a smooth Hermitian metric $\|\cdot\|_j$, for $j=1,\hdots,\kappa$. Letting $s = s_1 \otimes \cdots \otimes s_\kappa$, we have that $\omega \in \mathcal{A}^{\kappa,\kappa}_{s,\|\cdot\|}(X)$, where $\|\cdot\|$ is the smooth Hermitian metric on $L = L_1 \otimes \cdots \otimes L_\kappa$ satisfying $\|s\|^2 = \|s_1\|_1^2 \cdots \|s_\kappa\|^2_\kappa$. Under the assumption that $s_{j_1},\hdots,s_{j_k}$ is a locally complete intersection for each $1 \leq j_1 < \cdots < j_k \leq \kappa$, we give a natural meaning to products of the form
\begin{equation}
    \label{eq:product}
    \mu_{\ell_1}^{\|s_1\|_1}(\omega_1) \wedge \cdots \wedge \mu_{\ell_\kappa}^{\|s_\kappa\|_\kappa}(\omega_\kappa),
\end{equation}
where $(\ell_1,\hdots,\ell_\kappa)\in \mathbb{Z}_{\geq-1}^\kappa$, see \Cref{def:product} and \Cref{prop:3} below. These products turn out to be commutative and associative, and in the special case when $\mathrm{div}(s)$ is a normal crossings divisor they are, in fact, locally just sums of tensor products of currents (up to multiplication by smooth forms). Our main result says that the currents $\mu_\ell^{\|s\|}(\omega)$, associated to $\omega = \omega_1 \wedge \cdots \wedge \omega_\kappa$, are sums of products of the form \cref{eq:product}.

\medskip 

To formulate the main result in a precise way, we use the notion of a \textit{current-valued meromorphic function}. A current-valued meromorphic function is a mapping $\lambda \mapsto \mu(\lambda)$ defined on some domain $\Omega \subset \C$, such that for any test form $\xi$, $\langle \mu(\lambda),\xi\rangle$ is a meromorphic function on $\Omega$. For more details, see Section 2 below. In this terminology \cref{eq:laurentexpansion} says that
\begin{equation}
    \label{eq:laurentmap}
    \|s\|^{2\lambda} \omega = \sum\limits_{j=-\kappa}^\infty \lambda^j \mu_j^{\|s\|}(\omega),
\end{equation}
as current-valued meromorphic functions in a neighborhood of $\lambda = 0$.

Now we are ready to formulate the main result of this article.
\begin{theorem}
    \label{thm:1}
    Let $s_j \colon X \rightarrow L_j$ be a holomorphic section of a holomorphic line bundle $L_j \rightarrow X$ for $j=1,\hdots,\kappa$, where $1 \leq \kappa \leq n$, such that $s_{j_1},\hdots,s_{j_\ell}$ is a locally complete intersection for each $1\leq j_1 < \cdots < j_\ell \leq \kappa$. Let $\|\cdot\|_j$ be a smooth Hermitian metric on $L_j$ and let $\|\cdot\|$ be the smooth Hermitian metric on $L_{1} \otimes \cdots \otimes L_{\kappa}$ satisfying $\|s \|^2 = \|s_{1}\|_{1}^2 \cdots \|s_{\kappa}\|_{\kappa}^2$, where $s = s_1 \otimes \cdots \otimes s_{\kappa}$. We have the following equality of current-valued meromorphic functions,
    \begin{equation}
        \label{eq:mainresult}
        \sum\limits_{\ell=-\kappa}^\infty \lambda^\ell \mu_{\ell}^{\|s\|}(\omega_1 \wedge \cdots \wedge \omega_\kappa) = \Bigg( \sum\limits_{\ell_1 = -1}^\infty \lambda^{\ell_1} \mu_{\ell_1}^{\|s_{1}\|_1}(\omega_1) \Bigg) \wedge \cdots \wedge \Bigg( \sum\limits_{\ell_\kappa = -1}^\infty \lambda^{\ell_\kappa} \mu_{\ell_\kappa}^{\|s_{\kappa}\|_\kappa}(\omega_\kappa) \Bigg),
    \end{equation}
    in a neighborhood of $\lambda = 0$, where $\omega_j = \bar{\partial}\log\|s_j\|_j^2 \wedge \partial \log\|s_j\|_j^2$ for $j=1,\hdots,\kappa$.
\end{theorem}
For instance, \Cref{thm:1}, together with \cref{eq:mu-1}, gives the following decomposition formula for the coefficient $\mu_{-\kappa}^{\|s\|}(\omega_1\wedge \cdots \wedge \omega_\kappa)$ of the leading term of the left-hand side of \cref{eq:mainresult},
\begin{equation}
    \label{eq:intersectionlelongcurrent}
    \mu_{-\kappa}^{\|s\|}(\omega_1\wedge \cdots \wedge \omega_\kappa) = (2\pi i)^{\kappa} [\mathrm{div}(s_1)] \wedge \cdots \wedge [\mathrm{div}(s_\kappa)].
\end{equation}
The product of Lelong currents associated to a locally complete intersection of divisors in the right-hand side of \cref{eq:intersectionlelongcurrent} is standard in complex geometry, and can be obtained by any reasonable regularization of the factors if $X$ is a complex manifold. The resulting current is the Lelong current associated to the proper intersection $\mathrm{div}(s_1) \cdots \mathrm{div}(s_\kappa)$.

\medskip

\Cref{thm:1} allows us, at least in principle, to reduce the calculation of the finite part for any $\omega\in\mathcal{A}_s(X)$ to calculating products \cref{eq:product} where $\omega_j$ are elementary quasi-meromorphic forms. To see this, first note that by Hironaka's theorem, we can always find a modification $\pi \colon \widetilde{X} \rightarrow X$ such that $\widetilde{X}$ is smooth and $\pi^* s$ defines a normal crossings divisor with smooth irreducible components. Moreover, it turns out that $\pi^* \omega \in \mathcal{A}^{n,n}_{\pi^* s}(\widetilde{X})$ and $\mu_0^{\|s\|}(\omega) = \pi_* \mu_0^{\|\pi^* s\|}(\pi^*\omega)$. Thus, cf. \cref{eq:finitepartdef}, 
\begin{equation}
    \label{eq:finitepartmodification}
    \mathrm{fp}\int_X \omega = \langle \mu_0^{\|s\|}(\omega),1\rangle = \langle \pi_* \mu_0^{\|\pi^* s\|}(\pi^*\omega),1\rangle = \langle \mu_0^{\|\pi^* s\|}(\pi^*\omega),\pi^* 1\rangle  = \mathrm{fp}\int_{\widetilde{X}} \pi^*\omega.
\end{equation}
Consequently, for the purposes of computing finite parts, we can assume that $X$ is a smooth manifold and $s \colon L \rightarrow X$ defines a normal crossings divisor $D$. The second observation is that $\mathrm{fp}\int_X \omega$ depends only on the de Rham class of the current $\mu_0^{\|s\|}(\omega)$. Therefore, we can look for representatives of $[\mu_0^{\|s\|}(\omega)]_{\mathrm{dR}}$ that facilitate computation of the finite part. 

\medskip

The third observation is that for any $\omega \in \mathcal{A}^{n,n}_s(X)$ there are $\omega_\ell^{\mathrm{tame}} \in \mathcal{A}^{n,n}_s(X)$, for $\ell=0,\hdots,n$, such that
\begin{equation}
    \label{eq:tamecohomology}
    [\mu_0^{\|s\|}(\omega)]_{\mathrm{dR}} = \sum\limits_{\ell = 0}^n (-1)^\ell [\mu_{-\ell}^{\|s\|}(\omega_\ell^{\mathrm{tame}})]_{\mathrm{dR}},
\end{equation}
see \Cref{prop:1} and the paragraph following it. Here $\omega^{\mathrm{tame}}_\ell$ has \textit{tame singularities} along $D$, which means that locally there are holomorphic coordinates $z = (z_1,\hdots,z_n)$ such that
\[
    \omega^{\mathrm{tame}}_\ell = \frac{\widetilde{\omega}}{|z_1\cdots z_\kappa|^2},
\]
for some $0 \leq \kappa \leq n$, where $\widetilde{\omega}$ is a smooth top form.

The fourth observation, see \Cref{cor:1} below, is that any quasi-meromorphic top form $\omega^{\mathrm{tame}}$ with tame singularities along $D$ is of the form
\begin{equation}
    \label{eq:tamedecomposition}
    \omega^{\mathrm{tame}} = \sum_{|J|=0}^{\min\{n,q\}}\widetilde{\omega}_J \wedge \bigwedge_{j\in J} \partial \log\|s_j\|_j^2 \wedge \bar{\partial}\log\|s_j\|_j^2,
\end{equation}
where $\widetilde{\omega}_J$ are smooth forms and $s_j$ are holomorphic sections of Hermitian line bundles $(L,\|\cdot\|_j)$ such that $\mathrm{div}(s_j)$, for $j=1,\hdots,q$, are the irreducible components of $D$.

Finally, \Cref{thm:1} reduces the computation of $\mu_\ell^{\|s\|}(\omega^{\mathrm{tame}})$ to calculating products \cref{eq:product} where $\omega_j$ are elementary quasi-meromorphic forms.

\medskip

In Section 5 we prove \cref{eq:tamecohomology,eq:tamedecomposition} and fully detail an algorithm for computing finite parts based on these formulas. In Section 6 we apply the algorithm to compute finite parts of a certain quasi-meromorphic top form on $\mathbb{P}^n$ with tame singularities along the normal crossings divisor defined by $Z_0 \cdots Z_n \colon \mathbb{P}^n \rightarrow \mathcal{O}(n+1)$, for $n = 2$ and $n=3$. The form in question is given by
\[
    \omega = \frac{1}{n!}\frac{1}{\|Z_0\cdots Z_n\|^2} \omega_{\mathrm{FS}}^n,
\]
where $\omega_{\mathrm{FS}}$ is the Fubini--Study form, and $\|\cdot\|$ is the Fubini--Study metric on $\mathcal{O}(n+1)$ which we also use for the regularization of $\int_{\mathbb{P}^n}\omega$. In coordinates $z = (z_1,\hdots,z_n)$ on the standard affine piece of $\mathbb{P}^n$, we have that
\[
    \omega = \frac{i^n}{2^n} \frac{\d z_1 \wedge \d \bar{z}_1 \wedge \cdots \wedge \d z_n \wedge \d \bar{z}_n}{|z_1 \cdots z_n|^2}.
\]
%
We find that
\[
    \mathrm{fp}\int_{\mathbb{P}^2} \omega = -9\pi^2 \zeta(2)\qquad \text{and}\qquad \mathrm{fp}\int_{\mathbb{P}^3} \omega = 80 \pi^3 \zeta(3).
\]
Moreover, we provide an alternative ad-hoc formula for $\mathrm{fp}\int_{\mathbb{P}^n} \omega$, see \cref{eq:adhocformula} below, which holds for general $n$, and use it to compute the corresponding finite part for $n=4$ and $n=5$. We obtain the following:
\[
    \mathrm{fp}\int_{\mathbb{P}^4} \omega = -150\pi^4 \zeta(4)\qquad \text{and} \qquad \mathrm{fp}\int_{\mathbb{P}^5} \omega = 252\pi^5 \big( 37\zeta(5) - 25\zeta(2)\zeta(3)\big).
\]
We note that the resulting finite parts are all $\Q$-linear combinations of (multiple) zeta-values, multiplied by a factor of $\pi^n$.

Let us conclude this introduction by pointing out some intriguing connections to Deligne's and Beilinson's conjectures concerning the relations between periods and $L$-functions. 
Recall that motivic $L$-functions are a vast number-theoretic generalization of the Riemann zeta function. 
According to conjectures of Deligne and Beilinson, all central values of motivic $L$-functions can be expressed as periods of algebraic forms, defined over a number field, see, e.g., \cite{Z1} and \cite[Section 3]{KZ1}. It seems therefore natural to ask wether principal values or finite parts of divergent periods, in special situations, may also be expressed in terms of special values of $L$-functions, e.g., (multiple) zeta values, when using a canonical or algebraic metric in the regularization procedure. In this direction, we provide the aforementioned examples, which also serve to demonstrate the calculus developed in this paper.

\section{Preliminaries}

\subsection{Current-valued meromorphic functions}

Let $X$ be a reduced analytic space of pure dimension $n$. Recall that a current-valued holomorphic function is a mapping $\mu \colon \Omega \rightarrow \mathscr{D}'(X)$, defined on some domain $\Omega\subseteq \C$, such that, for any $\xi \in \mathscr{D}(X)$, $\langle \mu(\lambda),\xi \rangle$ is a holomorphic function in $\Omega$. Here $\mathscr{D}(X)$ denotes the space of test forms on $X$ and $\mathscr{D}'(X)$ the space of currents. The derivative $\d \mu / \d \lambda$ of a current-valued holomorphic function is the current-valued holomorphic function defined by
\[
    \langle \tfrac{\d}{\d \lambda}\mu(\lambda), \xi \rangle = \frac{\d}{\d \lambda} \langle \mu(\lambda), \xi \rangle,
\]
for any test form $\xi \in \mathscr{D}(X)$. The fact that $\frac{\d}{\d \lambda} \mu(\lambda)$ is a current for each $\lambda$ follows by standard arguments. The ability to take derivatives of current-valued holomorphic functions implies, in particular, that in a neighborhood of any point $\lambda_0 \in \Omega$, we can identify a given current-valued holomorphic function with its Taylor series,
\[
    \mu(\lambda) = \sum\limits_{j=0}^\infty \frac{(\lambda-\lambda_0)^j}{j!} \frac{\d^j \mu}{\d \lambda^j}(\lambda_0).
\]
A \textit{current-valued meromorphic function} is a mapping $\lambda \mapsto \mu(\lambda)$ defined on some domain $\Omega\subseteq\C$ that is a current-valued holomorphic function on the complement of a countable number of isolated points in $\Omega$, and such that for any $\xi \in \mathscr{D}(X)$, $\langle\mu(\lambda),\xi\rangle$ is a meromorphic function on $\Omega$. It follows by standard arguments that in a neighborhood of any point $\lambda_0 \in \Omega$, we can identify a current-valued meromorphic function $\mu(\lambda)$ with its \textit{Laurent series},
\[
    \mu(\lambda) = \sum\limits_{j=-\kappa}^\infty (\lambda-\lambda_0)^j \mu_j.
\]
Here $\mu_j$ is the current such that $\langle \mu_j,\xi \rangle$ is the $j$\textsuperscript{th} coefficient of the Laurent series of $\langle\mu(\lambda),\xi\rangle$ about $\lambda_0$, i.e., 
\begin{equation}
    \label{eq:residue}
    \langle \mu_j, \xi \rangle = \underset{\lambda = \lambda_0}{\mathrm{Res}}\Big\{\lambda^{-(j + 1)}\langle \mu(\lambda),\xi \rangle \Big\} = \frac{1}{(j+\kappa)!}\Big\langle \frac{\d^{j+\kappa}}{\d \lambda^{j+\kappa}}\big(\lambda^\kappa \mu(\lambda)\big),\xi\Big\rangle\Big|_{\lambda = \lambda_0},
\end{equation}
for $j \geq -\kappa$.

\subsection{Properties of $\mu_j^{\| s \|}(\omega)$}

Here we collect some properties of the currents $\mu_j^{\|s\|}(\omega)$ associated to a quasi-meromorphic form $\omega \in \mathcal{A}_{s,\|\cdot\|}(X)$ on a reduced analytic space $X$ of pure dimension $n$. As in the introduction, here $s \colon X \rightarrow E$ is a holomorphic section of a holomorphic vector bundle equipped with a Hermitian metric $\|\cdot\|$. See \cite{S1} for proofs.

For any $j$, $\mu_j^{\|\cdot\|}(\omega)$ is linear on smooth forms, that is, if $\alpha$ and $\beta$ are smooth forms then we have that
\begin{equation}
    \label{eq:linearity}
    \mu_j^{\|s\|}(\omega\wedge(\alpha + \beta)) = \mu_j^{\|s\|}(\omega) \wedge \alpha +  \mu_j^{\|s\|}(\omega) \wedge \beta.
\end{equation}
Furthermore, for $\omega_1,\omega_2 \in \mathcal{A}_{s,\|\cdot\|}(X)$, we have that
\begin{equation}
    \label{eq:additivity}
    \mu_j^{\|s\|}(\omega_1 + \omega_2) = \mu_j^{\|s\|}(\omega_1) + \mu_j^{\|s\|}(\omega_2).
\end{equation}
If $\eta \in \mathcal{A}_{s,\|\cdot\|}(X)$, then $\d \eta, \tfrac{\d \|s\|^2}{\|s\|^2} \wedge \eta \in \mathcal{A}_{s,\|\cdot\|}(X)$. Moreover, for each $j$ we have that
\begin{equation}
    \label{eq:thm4.1ii}
    \d \mu_j^{\|s\|}(\eta) = \mu_j^{\|s\|}(\d \eta) + \mu_{j-1}^{\|s\|} \bigg( \frac{\d \|s\|^2}{\|s\|^2} \wedge \eta\bigg).
\end{equation}

\section{Elementary quasi-meromorphic forms}

Let $X$ be a reduced analytic space of pure dimension $n$, and suppose that $s \colon X \rightarrow L$ is a holomorphic section of a Hermitian line bundle. As mentioned above, to a quasi-meromorphic form $\omega \in \mathcal{A}_{s,\|\cdot\|}^{p,q}(X)$ we associate the current-valued meromorphic function $\lambda \mapsto \|s\|^{2\lambda} \omega$ on $\C$, given by
\[
    \langle \|s\|^{2\lambda} \omega, \xi \rangle = \int_{X} \|s\|^{2\lambda} \omega \wedge \xi,\quad \xi \in \mathscr{D}(X).
\]
In a neighborhood of $\lambda = 0$, we identify $\|s\|^{2\lambda} \omega$ with its Laurent series, see \cref{eq:laurentmap}. Recall that $\mu_j^{\|s\|}(\omega)$ is the $j$\textsuperscript{th} coefficient of the Laurent series expansion of $\|s\|^{2\lambda}\omega$ about $\lambda = 0$. We have the following lemma describing $\mu_j^{\|s\|}(\omega)$ when $\omega$ is an elementary quasi-meromorphic form.
\begin{lemma}
    \label{lem:1}
    Let $s \colon X \rightarrow L$ be a holomorphic section of a holomorphic line bundle $L$ equipped with a smooth Hermitan metric $\|\cdot\|$, and let $\omega = \bar{\partial}\log\|s\|^2 \wedge \partial \log\|s\|^2$. For each $\ell \geq -1$, the current-valued map
    \begin{equation}
        \label{eq:Ljl}
        L_{\ell}(\lambda) = \frac{1}{(\ell+1)!} \bar{\partial}\Big( \big( \log\|s\|^2 \big)^{\ell+1} \|s\|^{2\lambda} \Big) \wedge \partial \log \|s\|^2,
    \end{equation}
    a priori defined and holomorphic for $\mathfrak{Re}\,\lambda \gg 0$, has an analytic continuation to a neighborhood of $\mathfrak{Re}\,\lambda\geq 0$. Moreover, 
    \begin{equation}
        \label{eq:elementarymuformula}
        \mu_\ell^{\|s\|}(\omega) = L_{\ell}(0) = \frac{1}{(\ell+2)!} \bar{\partial}\partial \big( \log\|s\|^2\big)^{\ell+2} + \frac{2\pi i}{(\ell+1)!}\big(\log\|s\|^2\big)^{\ell+1} c_{1}(L,\|\cdot\|),
    \end{equation}
    where $c_1(L,\|\cdot\|)$ is the Chern curvature of $(L,\|\cdot\|)$.
\end{lemma}
\begin{proof}
    For $\mathfrak{Re}\,\lambda \gg 0$, $L_\ell(\lambda)$ is clearly holomorphic and we have that
    \begin{equation}
        \label{eq:Lcomp}
        \begin{aligned}
            L_{\ell}(\lambda) &= \frac{1}{(\ell+1)!}\bar{\partial}\Big( \big( \log\|s\|^2 \big)^{\ell+1} \|s\|^{2\lambda} \Big) \wedge \partial \log \|s\|^2\\ 
            &= \frac{1}{(\ell+1)!}\frac{\d^{\ell+1}}{\d \lambda^{\ell+1}}\bigg( \bar{\partial}\|s\|^{2\lambda} \wedge \partial \log \|s\|^2 \bigg)\\
            &= \frac{1}{(\ell+1)!}\frac{\d^{\ell+1}}{\d \lambda^{\ell+1}}\bigg( \bar{\partial} \Big(\|s\|^{2\lambda} \partial \log \|s\|^2 \Big) - \|s\|^{2\lambda}\bar{\partial}\partial \log \|s\|^2 \bigg) \\
            &= \frac{1}{(\ell+1)!}\frac{\d^{\ell+1}}{\d \lambda^{\ell+1}}\bigg( \bar{\partial} \Big(\|s\|^{2\lambda} \partial \log \|s\|^2 \Big) + 2\pi i \|s\|^{2\lambda}c_1(L,\|\cdot\|) \bigg),
        \end{aligned}
    \end{equation}
    where the last step follows by the Poincaré--Lelong formula,
    \begin{equation}
        \label{eq:poincarelelong}
        \bar{\partial}\partial \log\|s\|^2 = 2\pi i [\mathrm{div}(s)] - 2 \pi i c_{1}(L,\|\cdot\|),
    \end{equation}
    and the fact that $\|s\|^{2\lambda}[\mathrm{div}(s)]=0$ for $\mathfrak{Re}\,\lambda>0$. It is clear that $\|s\|^{2\lambda}c_1(L,\|\cdot\|)$ is holomorphic in a neighborhood of $\mathfrak{Re}\,\lambda \geq 0$. The same is true for $\|s\|^{2\lambda} \partial\log\|s\|^{2\lambda}$ since $\partial\log\|s\|^2$ is locally integrable. It follows that $L_\ell(\lambda)$ extends, as a current-valued holomorphic function, to a neighborhood of $\mathfrak{Re}\,\lambda \geq 0$. Note also that
    \begin{align*}
        L_{\ell}(\lambda) &= \frac{1}{(\ell+1)!}\frac{\d^{\ell+1}}{\d \lambda^{\ell+1}}\bigg( \bar{\partial}\|s\|^{2\lambda} \wedge \partial \log \|s\|^2 \bigg) \\
        &= \frac{1}{(\ell+1)!}\frac{\d^{\ell+1}}{\d \lambda^{\ell+1}}\bigg( \lambda \|s\|^{2\lambda}\bar{\partial}\log\|s\|^2 \wedge \partial \log \|s\|^2 \bigg) \\
        &= \frac{1}{(\ell+1)!}\frac{\d^{\ell+1}}{\d \lambda^{\ell+1}}\bigg( \lambda \|s\|^{2\lambda}\omega \bigg).
    \end{align*}
    Thus, $\|s\|^{2\lambda}\omega$ extends to a current-valued meromorphic function in a neighborhood of $\mathfrak{Re}\,\lambda \geq 0$ with a pole at $\lambda = 0$ of order $\leq 1$. By setting $\lambda = 0$, it follows in veiw of \cref{eq:residue} that
    \begin{equation}
        \label{eq:Ll0}
        \begin{aligned}
            L_{\ell}(0) &= \frac{1}{(\ell+1)!}\frac{\d^{\ell+1}}{\d \lambda^{\ell+1}} \big(\lambda \|s\|^{2\lambda}\omega \big)\Big|_{\lambda = 0} \\
            &= \mu_{\ell}^{\|s\|}(\omega).
        \end{aligned}
    \end{equation}
    It remains to show the second equality in \cref{eq:elementarymuformula}. A first observation is that $(\log\|s\|^2)^{\ell}$ and $(\log\|s\|^2)^\ell \partial \log\|s\|^2$ are $L^1_{\mathrm{loc}}$ for any $\ell \in \mathbb{N}$. Thus, $\|s\|^{2\lambda}(\log\|s\|^2)^\ell \partial \log\|s\|^2$ and $\|s\|^{2\lambda}(\log\|s\|^2)^\ell c_1(L,\|\cdot\|)$ are holomorphic in a neighborhood of $\mathfrak{Re}\,\lambda \geq 0$, with 
    \[
        \|s\|^{2\lambda}\big(\log\|s\|^2\big)^\ell \partial \log\|s\|^2\Big|_{\lambda = 0} = \big(\log\|s\|^2\big)^\ell \partial \log\|s\|^2
    \]
    and
    \[
        \|s\|^{2\lambda}\big(\log\|s\|^2\big)^\ell c_{1}(L,\|\cdot\|)\Big|_{\lambda = 0} = \big(\log\|s\|^2\big)^\ell c_{1}(L,\|\cdot\|).
    \]
    By the last equality in \cref{eq:Lcomp}, we therefore have that
    \begin{align*}
        L_\ell(0) &= \frac{1}{(\ell+1)!}\frac{\d^{\ell+1}}{\d \lambda^{\ell+1}}\bigg( \bar{\partial} \Big(\|s\|^{2\lambda} \partial \log \|s\|^2 \Big) + 2\pi i \|s\|^{2\lambda}c_1(L,\|\cdot\|) \bigg)\bigg|_{\lambda = 0} \\
        &= \frac{1}{(\ell+1)!}\bigg( \bar{\partial} \Big(\big( \log\|s\|^2\big)^{\ell+1} \|s\|^{2\lambda} \partial \log \|s\|^2 \Big) \\
        &\qquad\qquad\qquad\qquad\qquad\ \ + 2\pi i \big( \log\|s\|^2\big)^{\ell+1} \|s\|^{2\lambda}c_1(L,\|\cdot\|) \bigg)\bigg|_{\lambda=0} \\
        &= \frac{1}{(\ell+1)!}\bigg( \bar{\partial} \Big(\big( \log\|s\|^2\big)^{\ell+1}\partial \log \|s\|^2 \Big) + 2\pi i \big( \log\|s\|^2\big)^{\ell+1} c_1(L,\|\cdot\|) \bigg) \\
        &= \frac{1}{(\ell+2)!} \bar{\partial} \partial \big( \log\|s\|^2\big)^{\ell+2} + \frac{2\pi i}{(\ell+1)!} \big( \log\|s\|^2 \big)^{\ell+1} c_1(L,\|\cdot\|).
    \end{align*}
    Thus, the last equality in \cref{eq:elementarymuformula} follows and the proof is complete.
\end{proof}

\section{Proof of \Cref{thm:1}}

In this section we define the product of currents \cref{eq:product} associated to locally complete intersections, and prove that it is commutative and associative. This is followed by a proof of our main result, \Cref{thm:1}.

Let $X$ be a reduced analytic space of pure dimension $n$. Let $s_j \colon X \rightarrow L_j$ be a holomorphic section of a Hermitian line bundle $(L_j, \|\cdot\|_j)$ for $j=1,\hdots,\kappa$, where $1 \leq \kappa \leq n$, such that $s_{j_1},\hdots,s_{j_\ell}$ is a locally complete intersection for each $1\leq j_1 < \cdots < j_\ell \leq \kappa$. Let $\omega_j = \bar{\partial}\log\|s_j\|_j^2 \wedge \partial \log\|s_j\|_j^2$.
\begin{definition}
    \label{def:product}
    For each $1\leq j_1 < \cdots < j_k \leq \kappa$ and $(\ell_1,\hdots,\ell_k) \in \mathbb{Z}_{\geq-1}^k$, we define the product $\bigwedge_{i=1}^k \mu_{\ell_i}^{\|s_{j_i}\|_{j_i}}(\omega_{j_i}) = \mu_{\ell_1}^{\|s_{j_1}\|_{j_1}}(\omega_{j_1}) \wedge \cdots \wedge \mu_{\ell_k}^{\|s_{j_k}\|_{j_k}}(\omega_{j_k})$ recursively by
    \begin{equation}
        \label{eq:recursivedef}
        \bigwedge_{i=1}^k \mu_{\ell_i}^{\|s_{j_i}\|_{j_i}}(\omega_{j_i}) := L^{j_1}_{\ell_1}(\lambda) \wedge \bigwedge_{i=2}^k \mu_{\ell_i}^{\|s_{j_i}\|_{j_i}}(\omega_{j_i})\bigg|_{\lambda = 0},
    \end{equation}
    where $L^{j_1}_{\ell_1}$ is given by \cref{eq:Ljl}, with $s$ replaced by $s_{j_1}$.
\end{definition}
This definition makes sense due to the following lemma.
\begin{lemma}
    \label{prop:3}
    The map $\lambda \mapsto L^{j_1}_{\ell_1}(\lambda) \wedge \bigwedge_{i=2}^k \mu_{\ell_i}^{\|s_{j_i}\|_{j_i}}(\omega_{j_i})$ appearing on the right-hand side of \cref{eq:recursivedef} is holomorphic in a neighborhood of $\lambda = 0$.
\end{lemma}
We also have the following.
\begin{proposition}
    \label{prop:4}
    The product $\bigwedge_{i=1}^k \mu_{\ell_i}^{\|s_{j_i}\|_{j_i}}(\omega_{j_i})$ is commutative and associative.
\end{proposition}
\begin{proof}[Proof of \Cref{prop:3} and \Cref{prop:4}]
    For any $\xi \in \mathscr{D}^{n-\kappa,n-\kappa}(X)$, consider the function
    \begin{equation}
        \label{eq:gdef}
        g \colon \boldsymbol{\lambda} = (\lambda_1, \hdots,\lambda_\kappa) \longmapsto \lambda_1 \cdots \lambda_\kappa \int_X \|s_1\|_1^{2\lambda_1} \cdots \|s_\kappa\|_\kappa^{2\lambda_\kappa} \omega_1 \wedge \cdots \wedge \omega_\kappa \wedge \xi.
    \end{equation}
    We claim that $g$ is holomorphic in a neighborhood of $\{ \boldsymbol{\lambda} \in \C^\kappa : \mathfrak{Re}\,\lambda_j \geq 0 \text{ for } j = 1,\hdots,\kappa \}$. Clearly $g$ is defined and holomorphic if $\mathfrak{Re}\,\lambda_j \gg 0$, for each $1\leq j \leq \kappa$. Moreover, for such $\lambda_j$ we have that
    \begin{align}
        \nonumber
        g(\boldsymbol{\lambda}) &= \lambda_1\cdots\lambda_\kappa \int_X  \bigg(\bigwedge_{j=1}^\kappa \|s_j\|_j^{2\lambda_j} \bar{\partial}\log\|s_j\|_j^{2}\wedge\partial\log\|s_j\|_j^2 \bigg) \wedge \xi \\
        \label{eq:gequality}
        &= \int_X  \bigg(\bigwedge_{j=1}^\kappa \bar{\partial}\|s_j\|_j^{2\lambda_j} \wedge \partial\log\|s_j\|_j^{2}\bigg) \wedge \xi.
    \end{align}
    By a partition of unity argument, we may suppose that $\xi$ has compact support in some chart $U$. Furthermore, we may assume that there is some $0 \leq k \leq \kappa$ such that, after a possible relabeling, $\|s_j\|_j^2 > 0$ in $U$ for each $k+1 \leq j \leq \kappa$, and such that, for each $j=1,\hdots,k$,
    \[
        \|s_j\|_j^2 = |f_j|^2 v_j,
    \]
    in $U$, where $v_j$ is a smooth strictly positive function and $f_j$ is a holomorphic function in $U$. Thus,
    \[
        \theta(\lambda_{k+1},\hdots,\lambda_\kappa) := \bigwedge_{j=k+1}^\kappa \bar{\partial}\|s_j\|_j^{2\lambda_j} \wedge \partial\log\|s_j\|_j^{2}
    \]
    is a smooth $(\kappa-k-1,\kappa-k-1)$-form on $U$ for each $(\lambda_{k+1},\hdots,\lambda_\kappa) \in \C^{\kappa-k-1}$, and we have that
    \begin{equation}
        \label{eq:glocal}
        \begin{aligned}
            g(\boldsymbol{\lambda}) &= \int_X  \bigg(\bigwedge_{j=1}^k \bar{\partial}(|f_j|^{2\lambda} v_j^\lambda) \wedge \Big( \frac{\d f_j}{f_j} + \frac{\partial v_j}{v_j} \Big) \bigg) \wedge \theta \wedge \xi \\
            &= \int_X  \bigg(\bigwedge_{j=1}^k \bigg[ \frac{\bar{\partial}|f_j|^{2\lambda}}{f_j}v_j^{\lambda_j} + \frac{|f_j|^{2\lambda_j}}{f_j} \bar{\partial}v_j^{\lambda_j} \bigg] \wedge \Big( \d f_j + f_j\frac{\partial v_j}{v_j}\Big)\bigg) \wedge \theta \wedge \xi.
        \end{aligned}
    \end{equation}
    It follows that
    \begin{equation}
        \label{eq:gsum}
        g(\boldsymbol{\lambda}) = \sum_{I,J}\int_X \frac{\bar{\partial}|f_{I_1}|^{2\lambda_{I_1}} \wedge \cdots \wedge \bar{\partial}|f_{I_{p}}|^{2\lambda_{I_{p}}} |f_{J_{1}}|^{2\lambda_{J_{1}}} \cdots |f_{J_{q}}|^{2\lambda_{J_q}} }{f_1 \cdots f_k} \wedge \varphi_{I,J}(\boldsymbol{\lambda}),
    \end{equation}
    where the sum runs over all partitions $\{I,J\}$ of $\{1,\hdots,k\}$ into two subsets, where we allow both $I$ and $J$ to be empty separately, and where, for each $\{I,J\}$,
    \[
        \varphi_{I,J}(\boldsymbol{\lambda}) = \pm \bigg( \bigwedge_{i\in I}v_i^{\lambda_i}\Big( \d f_{i} + f_i \frac{\partial v_i}{v_i} \Big) \bigg) \wedge  \bigg( \bigwedge_{j\in J} \bar{\partial}v_j^{\lambda_j} \wedge \Big( \d f_j + f_j \frac{\partial v_j}{v_j}\Big) \bigg) \wedge \theta \wedge \xi.
    \]
     Notice that $\varphi_{I,J}$ is a smooth $(n,n-|I|)$-form with compact support in $U$ for each $\boldsymbol{\lambda} \in \C^\kappa$, and holomorphic in $\boldsymbol{\lambda}$.

    In \cite{SK1} it is shown that for a holomorphic mapping $h=(h_1,\hdots,h_{p+q}) \colon X \rightarrow \C^{p+q}$ defining a complete intersection, a positive integer $N$ and any test form $\varphi \in \mathscr{D}^{n,n-p}(X)$, the map
    \begin{equation}
        \label{eq:håkan}
        \boldsymbol{\lambda} \mapsto \int_X \frac{\bar{\partial}|h_1|^{2\lambda_1} \wedge \cdots \wedge \bar{\partial} |h_p|^{2\lambda_p} |h_{p+1}|^{2\lambda_{p+1}} \cdots |h_{p+q}|^{2\lambda_{p+q}}}{h_1^N \cdots h_{p+q}^N} \wedge \varphi
    \end{equation}
    is analytic in a neighborhood of the half-space $\{\boldsymbol{\lambda} \in \C^{p+q} : \mathfrak{Re}\,\lambda_j \geq 0, 1 \leq j \leq p+q \}$. Since the mapping $(f_{1},\hdots,f_{k}) \colon U \rightarrow \C^k$ by assumption defines a complete intersection, our claim follows in view of \cref{eq:gsum}.

    \medskip

    Now, cf. \cref{eq:gequality} and the first equality in \cref{eq:Lcomp}, for any $(\ell_1,\hdots,\ell_\kappa) \in \mathbb{Z}_{\geq -1}^\kappa$,
    \begin{equation}
        \label{eq:gderiv}
        \begin{aligned}
            \frac{\partial^{\ell_1 + \cdots + \ell_\kappa + \kappa}}{\partial \lambda_1^{\ell_1+1} \cdots \partial \lambda_\kappa^{\ell_\kappa+1}} g(\boldsymbol{\lambda}) &= \int_X \bigg(\bigwedge_{j=1}^\kappa \bar{\partial}\Big( \big(\log\|s_j\|_j^2\big)^{\ell_j+1} \|s_j\|_j^{2\lambda_j}\Big) \wedge \partial \log \|s_j\|_j^2 \bigg) \wedge \xi \\
            &= (\ell_1 + 1)! \cdots (\ell_\kappa+1)! \langle L_{\ell_1}^1(\lambda_1) \wedge \cdots \wedge L_{\ell_\kappa}^\kappa(\lambda_\kappa), \xi \rangle.
        \end{aligned}
    \end{equation}
    It follows that the right-hand side is holomorphic in a neighborhood of the half space $\{ \boldsymbol{\lambda} \in \C^\kappa : \mathfrak{Re}\,\lambda_j \geq 0, 1 \leq j \leq\kappa \}$.

    We can unravel the definition of $\mu_{\ell_1}^{\|s_1\|_1}(\omega_1) \wedge \cdots \wedge \mu_{\ell_\kappa}^{\|s_\kappa\|_\kappa}(\omega_\kappa)$ as follows: By first setting $\lambda_\kappa =0$ in $L_{\ell_1}^1(\lambda_1) \wedge \cdots \wedge L_{\ell_\kappa}^\kappa(\lambda_\kappa)$, keeping $\mathfrak{Re}\,\lambda_j \gg 0$ for $1\leq j \leq \kappa-1$, we obtain $L_{\ell_1}^1(\lambda_1) \wedge \cdots \wedge L_{\ell_{\kappa-1}}^{\kappa-1}(\lambda_{\kappa-1})\wedge \mu_{\ell_\kappa}^{\|s_\kappa\|_\kappa}(\omega_\kappa) $, in view of \cref{eq:elementarymuformula}. Repeating the process for $\lambda_{\kappa-1}$, and so on, in view of \cref{eq:elementarymuformula} and \cref{eq:gderiv}, we obtain
    \begin{equation}
        \label{eq:mugrel}
        \langle \mu_{\ell_1}^{\|s_1\|_1}(\omega_1) \wedge \cdots \wedge \mu_{\ell_\kappa}^{\|s_\kappa\|_\kappa}(\omega_\kappa),\xi\rangle = \frac{1}{(\ell_1 + 1)! \cdots (\ell_\kappa + 1)!} \frac{\partial^{\ell_1 + \cdots + \ell_\kappa + \kappa} g}{\partial \lambda_1^{\ell_1+1} \cdots \partial \lambda_\kappa^{\ell_\kappa+1}} (0,\hdots,0).
    \end{equation}
    \Cref{prop:3} thus follows by the analyticity of $g(\boldsymbol{\lambda})$. Moreover, since the right-hand side of \cref{eq:mugrel} is independent of the order in which we set $\lambda_j$ to $0$, for $j=1,\hdots,\kappa$, \Cref{prop:4} follows.
\end{proof}
The above proof shows that $L^1_{\ell_1}(\lambda_1) \wedge \cdots \wedge L^{\kappa}_{\ell_\kappa}(\lambda_\kappa)$ is holomorphic in a neighborhood of $\{ \boldsymbol{\lambda} \in \C^\kappa : \mathfrak{Re}\,\lambda_j \geq 0, 1 \leq j \leq\kappa \}$. This may be of independent interest and we formulate it as a proposition.
\begin{proposition}
    \label{prop:Lprodholo}
    For any $(\ell_1,\hdots,\ell_\kappa) \in \mathbb{Z}^\kappa_{\geq-1}$, the product
    \[
        L^1_{\ell_1}(\lambda_1) \wedge \cdots \wedge L^{\kappa}_{\ell_\kappa}(\lambda_\kappa),
    \]
    where $L^j_{\ell_j}(\lambda_j)$ is given by \cref{eq:Ljl} for each $j=1,\hdots, \kappa$, is holomorphic in a neighborhood of the half space $\{ \boldsymbol{\lambda} \in \C^\kappa : \mathfrak{Re}\,\lambda_j \geq 0, 1 \leq j \leq\kappa \}$. In particular,
    \[
        L^1_{\ell_1}(\lambda_1) \wedge \cdots \wedge L^{\kappa}_{\ell_\kappa}(\lambda_\kappa)\Big|_{\boldsymbol{\lambda} = 0} = \mu^{\|s_1\|_1}_{\ell_1}(\omega_1) \wedge \cdots \wedge \mu^{\|s_\kappa\|_\kappa}_{\ell_\kappa}(\omega_\kappa).
    \]
\end{proposition}

Now we are ready to prove our main result.
\begin{proof}[Proof of \Cref{thm:1}]
    Recall from the proof of \Cref{prop:3} and \Cref{prop:4} that $g(\boldsymbol{\lambda})$, defined in \cref{eq:gdef}, is holomorphic in a neighborhood of $\{\boldsymbol{\lambda} \in \mathbb{C}^\kappa : \mathfrak{Re}\,\lambda_j \geq 0\text{ for } j = 1,\hdots,\kappa\}$. Let $\mathrm{diag} \colon \C \rightarrow \C^\kappa$ be the diagonal embedding. Then, in view of \cref{eq:gdef},
    \begin{equation}
        \label{eq:gdiag}
        g \circ \mathrm{diag} (\lambda) = \lambda^\kappa \int_X \big(\|s_1\|_1 \cdots \|s_\kappa\|_\kappa\big)^{2\lambda} \omega_1 \wedge \cdots \wedge \omega_\kappa \wedge \xi,
    \end{equation}
    which is holomorphic in a neighborhood of $\mathfrak{Re}\,\lambda \geq 0$ in $\C_\lambda$. Let
    \begin{equation}
        \Xi(\lambda) = \frac{g \circ \mathrm{diag}(\lambda)}{\lambda^\kappa},
    \end{equation}
    which is then a meromorphic function in a neighborhood of $\mathfrak{Re}\,\lambda \geq 0$ with a pole of order at most $\kappa$ at $\lambda=0$. It is clear that, for $\mathfrak{Re}\,\lambda > 0$, $\Xi(\lambda) = \langle \|s\|^{2\lambda}\omega, \xi \rangle$, where $\omega = \omega_1 \wedge \cdots \wedge \omega_\kappa$ and $\|s\|^2 = \|s_1\|^2_1 \cdots \|s_\kappa\|^2_\kappa$. Since $\langle \|s\|^{2\lambda}\omega, \xi \rangle$ extends to a meromorphic function on $\C_\lambda$ so too does $\Xi(\lambda)$. Furthermore, for $\ell \geq -\kappa$, we have by \cref{eq:residue} that
    \begin{align*}
        \langle \mu_{\ell}^{\|s\|}(\omega),\xi \rangle &= \underset{\lambda = 0}{\mathrm{Res}} \Big\{ \lambda^{-(\ell+1)} \Xi(\lambda) \Big\} \\
        &=  \underset{\lambda = 0}{\mathrm{Res}} \Big\{ \lambda^{-(\ell+\kappa +1)} g\circ \mathrm{diag}(\lambda) \Big\} \\
        &= \frac{1}{(\ell+\kappa)!}\frac{\d^{\ell+\kappa} (g \circ \mathrm{diag})}{\d \lambda^{\ell+\kappa}}(0).
    \end{align*}
    Since $g$ is holomorphic in each $\lambda_j$ in a neighborhood of $(\lambda_1,\hdots,\lambda_\kappa) = 0$, it follows by the chain rule that
    \[
         \langle \mu_{\ell}^{\|s\|}(\omega),\xi \rangle = \frac{1}{(\ell+\kappa)!} \sum\limits_{\substack{\ell_1,\hdots,\ell_\kappa \geq -1 \\ \ell_1 + \cdots + \ell_\kappa = \ell}} \frac{(\ell + \kappa)!}{(\ell_1 + 1)! \cdots (\ell_\kappa + 1)!} \frac{\partial^{\ell + \kappa} g}{\partial \lambda_1^{\ell_1+1} \cdots \partial \lambda_\kappa^{\ell_\kappa+1}} (0,\hdots,0).
    \]
    In view of \cref{eq:mugrel}, we see that
    \begin{equation}
        \label{eq:murelation}
        \langle \mu_{\ell}^{\|s\|}(\omega),\xi \rangle = \sum\limits_{\substack{\ell_1,\hdots,\ell_\kappa \geq -1 \\ \ell_1 + \cdots + \ell_\kappa = \ell}} \langle \mu_{\ell_1}^{\|s_1\|_1}(\omega_1) \wedge \cdots \wedge \mu_{\ell_\kappa}^{\|s_\kappa\|_\kappa}(\omega_\kappa), \xi \rangle.
    \end{equation}
    By expanding the right-hand side of \cref{eq:mainresult}, using the commutativity of the current product and collecting terms by order in $\lambda$, we see that \cref{eq:murelation} precisely gives us the desired equality, completing the proof.
\end{proof}
If $s_1 \otimes \cdots \otimes s_\kappa$ defines a normal crossings divisor, then the product of currents \cref{eq:product} is locally essentially just a tensor product. To see this, first off, for any point $p \in X$, for any $1\leq j\leq \kappa$ we can find holomorphic coordinates $z$ and a multi-index $m_j$ such that $\|s_j\|^{2} = |z^{m_j}|^2 v_j$ locally near $p$, where $v_j$ is a smooth positive function. By \cref{eq:elementarymuformula}, for any $\ell_j\geq-1$, we have that
\begin{equation}
    \label{eq:localpresentation}
    \begin{aligned}
        \mu_{\ell_j}^{\|s_j\|_j}(\omega_j) &= \frac{1}{(\ell_j+2)!} \bar{\partial}\partial \big( \log|z^{m_j}|^2 + \log v_j \big)^{\ell_j+2} \\
        &\qquad\qquad - \frac{1}{(\ell_j+1)!}\big(\log|z^{m_j}|^2 + \log v \big)^{\ell_j + 1} \bar{\partial}\partial \log v_j.
    \end{aligned}
\end{equation}
Now, by the locally complete intersection assumption the multi-indices $m_j$ are disjoint. In view of \cref{eq:localpresentation}, it is then straight-forward to check that the product of $\mu_{\ell_j}^{\|s_j\|_j}(\omega_j)$ 
has a natural definition locally, as a sum of tensor products of currents multiplied by smooth forms. 
We will write
\[
    \mu_{\ell_1}^{\|s_1\|_1}(\omega_1) \otimes \cdots \otimes \mu_{\ell_\kappa}^{\|s_\kappa\|_\kappa}(\omega_\kappa)
\]
for this product which, a priori, only makes sense with respect to a choice of local coordinates. However, the following proposition says that this product coincides with \Cref{def:product}.
%
%
%
\begin{proposition}
    \label{prop:5}
    Suppose that $s_{j_1},\hdots,s_{j_\ell}$ is a locally complete intersection for each $1\leq j_1 < \cdots < j_\ell \leq \kappa$ and that $s_1 \otimes \cdots \otimes s_\kappa$ defines a normal crossings divisor. Then, for each $(\ell_1,\hdots,\ell_\kappa) \in \mathbb{Z}_{\geq-1}^\kappa$, 
    \begin{equation}
        \label{eq:productlocalequality}
        \mu_{\ell_1}^{\|s_1\|_1}(\omega_1) \wedge \cdots \wedge \mu_{\ell_\kappa}^{\|s_\kappa\|_\kappa}(\omega_\kappa) = \mu_{\ell_1}^{\|s_1\|_1}(\omega_1) \otimes \cdots \otimes \mu_{\ell_\kappa}^{\|s_\kappa\|_\kappa}(\omega_\kappa).
    \end{equation}
    %
\end{proposition}
\begin{proof}
    We prove \cref{eq:productlocalequality} by induction over $\kappa$. The case $\kappa = 1$ is obvious. Now, suppose that 
    \[
        \mu_{\ell_1}^{\|s_1\|_1}(\omega_1) \wedge \cdots \wedge \mu_{\ell_{\kappa-1}}^{\|s_{\kappa-1}\|_{\kappa-1}}(\omega_{\kappa-1}) = \mu_{\ell_1}^{\|s_1\|_1}(\omega_1) \otimes \cdots \otimes \mu_{\ell_{\kappa-1}}^{\|s_{\kappa-1}\|_{\kappa-1}}(\omega_{\kappa-1})
    \]
    holds. Recall that by \Cref{prop:Lprodholo}, $L^1_{\ell_1}(\lambda_1) \wedge \cdots \wedge L^\kappa_{\ell_\kappa}(\lambda_\kappa)$ is holomorphic in a neighborhood of the half-space $\{\boldsymbol{\lambda}\in\mathbb{C}^\kappa : \mathfrak{Re}\,\lambda_j \geq 0,\ 1\leq j \leq \kappa\}$. Moreover,
    \[
        \lim_{\boldsymbol{\lambda} \rightarrow 0} L^1_{\ell_1}(\lambda_1) \wedge \cdots \wedge L^\kappa_{\ell_\kappa}(\lambda_\kappa) = \mu_{\ell_1}^{\|s_1\|_1}(\omega_1) \wedge \cdots \wedge \mu_{\ell_\kappa}^{\|s_\kappa\|_\kappa}(\omega_\kappa),
    \]
    where the limit can be taken in any order. Thus,
    \begin{align*}
        \mu_{\ell_1}^{\|s_1\|_1}(\omega_1) \wedge \cdots \wedge \mu_{\ell_\kappa}^{\|s_\kappa\|_\kappa}(\omega_\kappa) &= L^1_{\ell_1}(\lambda_1) \wedge \cdots \wedge L^\kappa_{\ell_\kappa}(\lambda_\kappa)\Big|_{\boldsymbol{\lambda} = 0} \\
        &=  L^1_{\ell_1}(0) \wedge \cdots \wedge L^{\kappa-1}_{\ell_{\kappa-1}}(0) \wedge L^\kappa_{\ell_\kappa}(\lambda_\kappa) \Big|_{\lambda_\kappa = 0} \\
        &= \mu_{\ell_1}^{\|s_1\|_1}(\omega_1) \otimes \cdots \otimes \mu_{\ell_{\kappa-1}}^{\|s_{\kappa-1}\|_{\kappa-1}}(\omega_{\kappa-1}) \wedge L^\kappa_{\ell_{\kappa}}(\lambda_\kappa) \Big|_{\lambda_\kappa = 0}.
    \end{align*}
    By \Cref{lem:1},
    \[
        L^\kappa_{\ell_\kappa}(\lambda_\kappa) = \frac{1}{(\ell_\kappa+1)!} \bar{\partial}\Big( \big( \log\|s_\kappa\|_\kappa^2 \big)^{\ell+1} \|s_\kappa\|^{2\lambda_\kappa}_\kappa \Big) \wedge \partial \log \|s_\kappa\|_\kappa^2 \underset{\lambda_\kappa \rightarrow 0}{\longrightarrow} \mu_{\ell_\kappa}^{\|s_\kappa\|_\kappa}(\omega_\kappa). 
    \]
    Again, since $\mu_{\ell_j}^{\|s_j\|_j}(\omega_j)$ are locally of the form \cref{eq:localpresentation} where $m_j$, for $j=1,\hdots,\kappa$, are disjoint multi-indices, it follows that
    \[
        \mu_{\ell_1}^{\|s_1\|_1}(\omega_1) \otimes \cdots \otimes \mu_{\ell_{\kappa-1}}^{\|s_{\kappa-1}\|_{\kappa-1}}(\omega_{\kappa-1}) \wedge L^\kappa_{\ell_{\kappa}}(\lambda_\kappa) \Big|_{\lambda_\kappa = 0} = \mu_{\ell_1}^{\|s_1\|_1}(\omega_1) \otimes \cdots \otimes \mu_{\ell_\kappa}^{\|s_\kappa\|_\kappa}(\omega_\kappa),
    \]
    completing the proof.
\end{proof}

\section{Finite part algorithm}

Recall from the introduction, cf. \cref{eq:finitepartmodification}, that calculating the finite part of any quasi-meromorphic form on a reduced analytic space can be reduced to a calculation of the finite part of a quasi-meromorphic form $\omega$, on a manifold, with singularities along a normal crossings divisor. Here we detail the outline presented in the introduction of the calculation of the finite part of such an $\omega$. The proofs are collected in Section 5.2.

\subsection{Finite part algorithm for a quasi-meromorphic form singular along a normal crossings divisor}

Let $X$ be a complex $n$-dimensional manifold and $\omega \in \mathcal{A}^{n,n}_{s,\|\cdot\|}(X)$, where $s$ is a holomorphic section of a Hermitian line bundle $(L,\|\cdot\|)$ over $X$ and $D = \mathrm{div}(s)$ is a normal crossings divisor.
\begin{lemma}
    \label{prop:1}
    There are quasi-meromorphic forms $\omega^{\mathrm{tame}}$ and $\gamma$ such that
    \begin{equation}
        \label{eq:tameplusexact}
        \omega = \omega^{\mathrm{tame}} + \d \gamma,
    \end{equation}
    where $\omega^{\mathrm{tame}}$ has tame singularities along $D$.
\end{lemma}
Recall that tame singularities means that locally $\omega^{\mathrm{tame}} = \widetilde{\omega}/|z_1\hdots z_\kappa|^2$, where $\widetilde{\omega}$ is smooth. \Cref{prop:1} has implications for the finite part $\mathrm{fp}\int_X \omega$. Let $\omega = \omega_0^{\mathrm{tame}} + \d \gamma_0$ be a decomposition of $\omega$ as in \Cref{prop:1}. For any $j\geq -n$, by \cref{eq:additivity} and \cref{eq:thm4.1ii}
\begin{align*}
    \mu_{j}^{\|s\|}(\omega) &= \mu_{j}^{\|s\|}(\omega^{\mathrm{tame}}_0 + \d \gamma_0) \\
    &= \mu_{j}^{\|s\|}(\omega^{\mathrm{tame}}_0) + \mu_{j}^{\|s\|}(\d \gamma_0) \\
    &= \mu_{j}^{\|s\|}(\omega^{\mathrm{tame}}_0) + \d \mu_{j}^{\|s\|}(\gamma_0) - \mu_{j-1}^{\|s\|}\bigg( \frac{\d \|s\|^2}{\|s\|^2} \wedge \gamma_0\bigg).
\end{align*}
Applying \Cref{prop:1} again, this time to the quasi-meromorphic top form $\frac{\d \|s\|^2}{\|s\|^2} \wedge \gamma_0$, we get $\frac{\d \|s\|^2}{\|s\|^2} \wedge \gamma_0 = \omega_1^{\mathrm{tame}} + \d \gamma_1$, where $\gamma_1$ is quasi-meromorphic and where $\omega_1^{\mathrm{tame}}$ has tame singularities. Again, by \cref{eq:linearity} and \cref{eq:thm4.1ii}, we have that
\[
    \mu_{j-1}^{\|s\|}\bigg( \frac{\d \|s\|^2}{\|s\|^2} \wedge \gamma_0\bigg) = \mu_{j-1}^{\|s\|}(\omega^{\mathrm{tame}}_1) + \d \mu_{j-1}^{\|s\|}(\gamma_1) - \mu^{\|s\|}_{j-2}\bigg( \frac{\d \|s\|^2}{\|s\|^2} \wedge \gamma_1\bigg).
\]
Thus,
\[
    \mu_j^{\|s\|}(\omega) = \mu_j^{\|s\|}(\omega_0^{\mathrm{tame}}) - \mu_{j-1}^{\|s\|}(\omega_1^{\mathrm{tame}}) + \d \big( \mu_j^{\|s\|}(\gamma_0) - \mu_{j-1}^{\|s\|}(\gamma_1)\big) + \mu_{j-2}^{\|s\|}\bigg( \frac{\d \|s\|^2}{\|s\|^2} \wedge \gamma_1\bigg).
\]
Continuing iteratively, noting that $\mu_{-n-1}^{\|s\|}(\omega') = 0$ for any quasi-meromorphic form $\omega'$, we find quasi-meromorphic forms $\omega_\ell^{\mathrm{tame}}$ and $\gamma_\ell$, for $\ell = 2,\hdots,n+j$, where $\omega_\ell^{\mathrm{tame}}$ has tame singularities such that
\[
    \mu_{j}^{\|s\|}(\omega) = \sum_{\ell=0}^{n+{j}} (-1)^{\ell} \mu_{j-\ell}^{\|s\|}(\omega^{\mathrm{tame}}_\ell) + \d \Big(\mu_{j}^{\|s\|}(\gamma_0) - \mu_{j-1}^{\|s\|}(\gamma_1)+ \cdots + (-1)^{n+j}\mu_{-n}^{\|s\|}(\gamma_{n+j})\Big).
\]
For $j=0$ we obtain formula \cref{eq:tamecohomology} by passing to cohomology. Thus, by \cref{eq:tamecohomology},
\begin{equation}
    \label{eq:finitepartcalc}
    \mathrm{fp}\int_X\omega = \langle \mu_0^{\|s\|}(\omega),1\rangle = \sum\limits_{\ell=0}^n (-1)^\ell \langle \mu_{-\ell}^{\|s\|}(\omega_\ell^{\mathrm{tame}}),1\rangle.
\end{equation}

For a general quasi-meromorphic top form $\omega^{\mathrm{tame}}$ with tame singularities along a normal crossings divisor, the following result shows that $\omega^{\mathrm{tame}}$ is composed of elementary quasi-meromorphic forms. 
\begin{proposition}
    \label{cor:1}
    Let $\omega^{\mathrm{tame}}$ be a quasi-meromorphic top form with tame singularities along $D$. Then there are holomorphic sections $s_j$ of Hermitian line bundles $(L_j, |\cdot|_j)$, for $j=1,\hdots,q$, such that $s_{j_1}, \hdots, s_{j_{k}}$ is a locally complete intersection for each $1\leq j_1 < \cdots < j_{k}\leq q$, where $k \leq \min\{n,q\}$, and
    \begin{equation}
        \label{eq:omegadecomp}
        \omega^{\mathrm{tame}} = \sum_{\substack{J \subseteq \{1,\hdots,q\} \\ 0 \leq |J| \leq \mathrm{min}\{n,q\}}}\widetilde{\omega}_J \wedge \bigwedge_{j\in J} \partial \log|s_j|_j^2 \wedge \bar{\partial}\log|s_j|_j^2,
    \end{equation}
    where $\widetilde{\omega}_J$ is a smooth $(n-|J|, n-|J|)$-form on $X$.
\end{proposition}
In view of \cref{eq:linearity} and \cref{eq:additivity} we get
\begin{equation}
    \label{eq:mujtamedecomp}
    \mu_j^{\|s\|}(\omega^{\mathrm{tame}}) = \sum_{|J|=0}^{\min\{n,q\}}\widetilde{\omega}_J \wedge \mu_j^{\|s\|}(\omega_J),
\end{equation}
where $\omega_J = \bigwedge_{j\in J} \partial \log|s_j|_j^2 \wedge \bar{\partial}\log|s_j|_j^2$.

\Cref{thm:1} is not applicable directly here. One can assume that the $s_j$ in \Cref{cor:1} are such that $s=s_1 \otimes \cdots \otimes s_q$ and $L = L_1\otimes \cdots \otimes L_q$. In fact, the proof of \Cref{cor:1} gives sections $s_j$ of $L_j$ such that $\mathrm{div}(s_j)$ are the irreducible components of $D$. For suitable $k_j \in \mathbb{N}$ we then have that $s = s_1^{k_1} \otimes \cdots \otimes s_{q}^{k_q}$, possibly up to an invertible holomorphic function, and $L = L_1^{k_1} \otimes \cdots \otimes L_q^{k_q}$. With the induced metric on $L_j^{k_j}$ we have $\bar{\partial}\log|s_j^{k_j}|^2_{j} \wedge \partial\log|s^{k_j}|^2_{j} = k_j^2 \bar{\partial}\log|s_j|^2_{j} \wedge \partial\log|s|^2_{j}$. We can thus replace $s_j$ by $s_j^{k_j}$ in \cref{eq:omegadecomp} if $\widetilde{\omega}_J$ is replaced by $\widetilde{\omega}_J / \prod_{j\in J}k_j^2$.

Now, since $L = L_1\otimes \cdots \otimes L_q$, the metrics on $L_j$ induce a metric $|\cdot|$ on $L$. Then, Theorem 1.1 in \cite{S1} states that
\begin{equation}
    \label{eq:metricformula}
    \mu_j^{\|s\|}(\omega_J) = \sum\limits_{\ell=0}^{n+j} \frac{1}{\ell!}\bigg(\log\frac{\|s\|^2}{|s|^2}\bigg)^\ell\mu_{j-\ell}^{|s|}(\omega_J),
\end{equation}
relating the currents defined with respect to the two different metrics. \Cref{thm:1} is still not directly applicable unless $J=\{1,\hdots,q\}$. First we apply the following variation of Theorem 1.1 in \cite{S1}:
\begin{proposition}
    \label{prop:6}
    For any tuple $J=(J_1,\hdots,J_k)$ with $1\leq J_1 < \cdots < J_k \leq q$, let $s_J = s_{J_1} \otimes \cdots \otimes s_{J_k} \colon X \rightarrow \bigotimes_{j\in J}L_j =: L_J$. Let $|\cdot|_J$ be the smooth Hermitian metric on $L_J$ defined by $|s_J|_J^2 = |s_{J_1}|_{J_1}^2 \cdots |s_{J_k}|_{J_k}^2$. We have that
    \begin{equation}
         \label{eq:generalized_metric_dependence_formula}
        \mu_\ell^{|s|}(\omega_J) = \sum\limits_{\ell'=0}^{k+\ell}\frac{1}{\ell'!} \bigg( \log\frac{|s|^2}{|s_J|^2_J}\bigg)^{\ell'} \mu_{\ell -\ell'}^{|s_J|_J}(\omega_J),
    \end{equation}
    for each $\ell$.
\end{proposition}
The proposed products on the right-hand side of \cref{eq:generalized_metric_dependence_formula} are of currents, and thus not a priori defined. However, these products can be defined in analogy with \Cref{def:product}, see \cref{eq:newproduct} below.

\Cref{thm:1} is now applicable to the currents $\mu_j^{|s_J|_J}(\omega_J)$ and gives
\begin{equation}
    \label{eq:thm1application}
    \mu_j^{|s_J|_J}(\omega_J) = \sum\limits_{\substack{\ell_1, \hdots,\ell_{|J|}\geq-1 \\ \ell_1 + \cdots + \ell_{|J|} = j}}\bigwedge_{k=1}^{|J|} \mu_{\ell_k}^{|s_{J_k}|_{J_k}}\big(\partial\log|s_{J_k}|_{J_k}^2 \wedge \bar{\partial}\log|s_{J_k}|_{J_k}^2 \big).
\end{equation}
To summarize, $\mathrm{fp}\int_X \omega$ can now in principle be calculated in view of \cref{eq:finitepartcalc}, \cref{eq:mujtamedecomp} \cref{eq:metricformula}, \cref{eq:generalized_metric_dependence_formula}, \cref{eq:thm1application} and the explicit formula \cref{eq:elementarymuformula} for $\mu_\ell^{|s_j|_j}(\omega_j)$. To write the full formula, assume for simplicity that $\|\cdot\|$ is equal to the metric $|\cdot|$ induced by the metrics $|\cdot|_j$ on $L_j$ for $j=1,\hdots,q$. Then
\begin{equation}
    \label{eq:masterformula}
    \begin{aligned}
        \mu_{\ell}^{|s|}(\omega) &= \sum_{|J|=0}^{\min\{n,q\}}\widetilde{\omega}_J \wedge \sum\limits_{\ell' = 0}^{|J|+\ell} \frac{1}{\ell'}\bigg( \log\frac{|s|^2}{|s_J|^2_J}\bigg)^{\ell'} \mu_{\ell-\ell'}^{|s_J|_J}(\omega_J) \\
        &= \sum_{|J|=0}^{\min\{n,q\}} \widetilde{\omega}_J \wedge \sum\limits_{\ell' = 0}^{|J|+\ell} \frac{1}{\ell'!} \bigg( \log\frac{|s|^2}{|s_J|^2_J}\bigg)^{\ell'} \sum\limits_{\substack{\ell_1,\hdots,\ell_{|J|} \geq -1 \\ \ell_1 + \cdots + \ell_{|J|} = \ell - \ell'}} \bigwedge_{k=1}^{|J|} \mu_{\ell_k}^{|s_{J_k}|_{J_k}}(\omega_{J_k}).
    \end{aligned}
\end{equation}

\subsection{Proofs of \Cref{prop:1}, \Cref{cor:1} and \Cref{prop:6}}

\begin{proof}[Proof of \Cref{prop:1}]
    We can find an open covering $\{U_\alpha\}$ of $X$ such that, in each chart $U_\alpha$, there are holomorphic coordinates $z =(z_1,\hdots,z_n)$, and an integer $\kappa$ such that $V = \{z_1\cdots z_\kappa = 0\}$ (suppressing the dependence on the chart $U_\alpha$ in the notation). Let $\{\rho_\alpha\}$ be a partition of unity subordinate to $\{U_\alpha\}$, and write $\omega = \sum_\alpha \rho_\alpha \omega$. Since $\omega$ is quasi-meromorphic, $\rho_\alpha \omega$ is quasi-meromorphic and there are positive integers $m_1,\hdots,m_\kappa$ (again these depend on the chart $U_\alpha$) such that
    \[
        \rho_\alpha\omega = \frac{\psi \,\d z_1 \wedge \d \bar{z}_1 \wedge \cdots \wedge \d z_n \wedge \d \bar{z}_n}{|z_1|^{2m_1} \cdots |z_\kappa|^{2m_\kappa}},
    \]
    where $\psi$ is a smooth function with compact support in $U_\alpha$. Assuming $m_1>1$, we have that
    \begin{align*}
        \rho_\alpha\omega &= \frac{\d z_1}{z_1^{m_1}}\wedge \frac{\psi \,\d \bar{z}_1 \wedge \cdots \wedge \d z_n \wedge \d \bar{z}_n}{\bar{z}_1^{m_1} |z_2|^{2m_2}\cdots |z_\kappa|^{2m_\kappa}} \\
        &= -\frac{1}{m_1 - 1} \d \bigg( \frac{1}{z_1^{m_1-1}}\bigg)\wedge \frac{\psi \,\d \bar{z}_1 \wedge \cdots \wedge \d z_n \wedge \d \bar{z}_n}{\bar{z}_1^{m_1} |z_2|^{2m_2}\cdots |z_\kappa|^{2m_\kappa}} \\
        &= \d \bigg( \frac{-1}{m_1 - 1}\frac{\psi \,\d \bar{z}_1 \wedge \cdots \wedge \d z_n \wedge \d \bar{z}_n}{z_1^{m_1 - 1} \bar{z}_1^{m_1} |z_2|^{2m_2}\cdots |z_\kappa|^{2m_\kappa}} \! \bigg)\! +\frac{1}{m_1 - 1}\frac{\d\psi \wedge \d \bar{z}_1 \wedge \cdots \wedge \d z_n \wedge \d \bar{z}_n}{z_1^{m_1 - 1} \bar{z}_1^{m_1} |z_2|^{2m_2}\cdots |z_\kappa|^{2m_\kappa}},
    \end{align*}
    outside $V$. Note that the expression in brackets above is quasi-meromorphic. Moreover,
    \[
        \frac{1}{m_1 - 1}\frac{\d\psi \wedge \d \bar{z}_1 \wedge \cdots \wedge \d z_n \wedge \d \bar{z}_n}{z_1^{m_1 - 1} \bar{z}_1^{m_1} |z_2|^{2m_2}\cdots |z_\kappa|^{2m_\kappa}} = \frac{\widetilde{\psi}\,\d z_1 \wedge \d \bar{z}_1 \wedge \cdots \wedge \d z_n \wedge \d \bar{z}_n}{z_1^{m_1 - 1} \bar{z}_1^{m_1} |z_2|^{2m_2}\cdots |z_\kappa|^{2m_\kappa}},
    \]
    where $\widetilde{\psi} = (\partial \psi /\partial z_1)/(m_1-1)$ is smooth with compact support in $U_\alpha$. It is clear that we can repeat the above process $m_1-2$ times, as well as, in the same way, $m_j - 1$ times in $z_j$ for each $j = 2,\hdots,\kappa$. Analogously, it can be repeated $m_j-1$ times in $\bar{z}_j$ for each $j=1,\hdots,\kappa$. Doing this, and collecting all the exact terms in $\d \gamma_\alpha$, we find that
    \[
        \rho_\alpha \omega = \frac{\psi_0\, \d z_1 \wedge \d \bar{z}_1 \wedge \cdots \wedge \d z_n \wedge \d \bar{z}_n}{|z_1 \cdots z_\kappa|^2} + \d \gamma_\alpha,
    \]
    where
    \[
        \psi_0 = \frac{1}{(m_1-1)!^2 \cdots (m_\kappa-1)!^2} \frac{\partial^{2\sum_{j=1}^\kappa m_j - 2\kappa} \psi}{ \partial z_1^{m_1-1} \partial \bar{z}_1^{m_1-1} \cdots \partial z_\kappa^{m_\kappa-1}\partial \bar{z}_\kappa^{m_\kappa-1}}
    \]
    is a smooth function with compact support in $U_\alpha$ and $\gamma_\alpha$ is quasi-meromorphic. Clearly
    \[
        \omega^{\mathrm{tame}}_\alpha := \frac{\psi_0\, \d z_1 \wedge \d \bar{z}_1 \wedge \cdots \wedge \d z_n \wedge \d \bar{z}_n}{|z_1 \cdots z_\kappa|^2}
    \]
    has tame singularities along $V\cap U_\alpha$. Now, letting $\omega^{\mathrm{tame}}:=\sum_\alpha \omega_\alpha^{\mathrm{tame}}$ and $\gamma := \sum_\alpha \gamma_\alpha$, \cref{eq:tameplusexact} follows, finishing the proof.
\end{proof}

\medskip

The proof of \Cref{cor:1} relies on the following local statement:
\begin{lemma}
    \label{lem:}
    Suppose
    \[
        \omega = \frac{\d z_1 \wedge \d \bar{z}_1 \wedge \cdots \wedge \d z_\kappa \wedge \d \bar{z}_\kappa}{|z_1 \cdots z_\kappa|^2} \wedge \eta,
    \]
    where $\eta$ is a smooth $(n-\kappa,n-\kappa)$-form on $\C^n$. Let $e^{-\phi_1},\hdots,e^{-\phi_\kappa}$ be smooth positive functions on $\C^n$ and let $|\cdot|_\ell := |\cdot| e^{-\phi_\ell/2}$ for $\ell=1,\hdots,\kappa$, where $|\cdot|$ is the standard absolute value on $\C$. Locally, in a neighborhood of $\{z_1 = \cdots = z_\kappa = 0\}$, we can find a smooth $(n-\kappa,n-\kappa)$-form $\widetilde{\omega}$ such that
    \[
        \omega = \bar{\partial} \log |z_1|_1^2 \wedge \partial \log|z_1|_1^2 \wedge \cdots \wedge  \bar{\partial} \log |z_\kappa|_\kappa^2 \wedge \partial \log |z_\kappa|_\kappa^2 \wedge \widetilde{\omega}.
    \]
\end{lemma}
\begin{proof}
    Let $\Theta = \bigwedge_{\ell=1}^{\kappa} \partial\log|z_\ell|_\ell^2 \wedge \bar{\partial}\log|z_\ell|_\ell^2 $. Expanding
    \begin{align}
        \nonumber
        \Theta &= \bigwedge_{\ell=1}^{\kappa} \partial \log\big( |z_\ell|^2 e^{-\phi_\ell}\big) \wedge \bar{\partial } \log \big( |z_\ell|^2 e^{-\phi_\ell} \big) \\
        \nonumber
        &= \bigwedge_{\ell=1}^{\kappa} \bigg(\frac{\d z_\ell}{z_\ell} - \partial \phi_\ell \bigg) \wedge \bigg(\frac{\d \bar{z}_\ell}{\bar{z}_\ell} - \bar{\partial}\phi_\ell \bigg) \\
        \nonumber
        &= \frac{1}{|z_1 \cdots z_\kappa|^2}\bigwedge_{\ell=1}^{\kappa} \Big(\d z_\ell \wedge \d \bar{z}_\ell - z_\ell \,\partial \phi_\ell \wedge \d \bar{z}_\ell - \bar{z}_\ell \,\d z_\ell \wedge \bar{\partial} \phi_\ell + |z_\ell|^2 \partial\phi_\ell \wedge \bar{\partial}\phi_\ell \Big) \\
        \label{eq:psicombination}
        &= \sum_{\substack{J,K\subset\{1,\hdots,n\}\\|J|,|K|=\kappa}} \psi_{J,K} \d z_J \wedge \d \bar{z}_K.
    \end{align}
    It is clear that the coefficient $\psi_{J_0,K_0}$ of the $(\kappa,\kappa)$-form $\d z_1 \wedge \d \bar{z}_1 \wedge \cdots\wedge \d z_\kappa \wedge \d \bar{z}_\kappa$ in the above expression satisfies
    \[
        |z_1 \cdots z_\kappa|^2 \psi_{J_0,K_0} =  1 + \mathcal{O}(|z_1|) + \cdots + \mathcal{O}(|z_\kappa|).
    \]
    Thus, we can find a neighborhood of $\{z_1 = \cdots = z_\kappa = 0\}$ where $|z_1\cdots z_\kappa|^2 \psi_{J_0,K_0}$ is non-zero. Now, let
    \[
        \widetilde{\omega} = \frac{\eta}{|z_1 \cdots z_\kappa|^2\psi_{J_0,K_0}},
    \]
    which is then a smooth form in this neighborhood. Notice that $\eta$ contains all $\d z_j \wedge \d \bar{z}_j$ for $j=\kappa+1,\hdots,n$. Then we have that
    \[
        \Theta \wedge \widetilde{\omega} = \frac{\d z_1 \wedge \d \bar{z}_1 \wedge \cdots \wedge \d z_\kappa \wedge \d \bar{z}_\kappa}{|z_1 \cdots z_\kappa|^2} \wedge \eta = \omega,
    \]
    since the only term in \cref{eq:psicombination} that contributes is $\psi_{J_0,K_0} \d z_1 \wedge \d \bar{z}_1 \wedge \cdots \wedge \d z_\kappa \wedge \d \bar{z}_\kappa$.
\end{proof}
\begin{proof}[Proof of \Cref{cor:1}]
    Let $D_j$, for $j=1,\hdots,q$, be the irreducible components of $D$ and let $s_j$ be holomorphic sections of $L_j$ such that $\mathrm{div}(s_j) = D_j$. Equip $L_j$ with Hermitian metrics $|\cdot|_j$. Since $D$ is a normal crossings divisor with smooth irreducible components, it follows that $s_{j_1}, \hdots, s_{j_k}$ is a locally complete intersection for each $1\leq j_1 < \cdots < j_{k} \leq q$ with $k\leq \min\{n,q\}$.
    
    For each point $p \in \{s_1\otimes \hdots \otimes s_q=0\}$, there is some number $k_p \leq \min\{n,q\}$ such that, after a possible relabeling of the sections $s_j$, $p \in \{s_1 = \cdots = s_{k_p} = 0\} \cap \{s_{k_p+1}, \hdots, s_{q} \neq 0 \}$. Let $U$ be an open neighborhood of $p$ such that $U \cap \{ s_j = 0 \} = \emptyset$, for $j = k_p+1,\hdots,q$. Since $\omega$ has tame singularities, we can find holomorphic coordinates $z = (z_1,\hdots,z_n)$ in $U$ such that
    \[
        \omega = \frac{\d z_1\wedge \d \bar{z}_1 \wedge \cdots \wedge \d z_{k_p} \wedge \d \bar{z}_{k_p}}{|z_1 \cdots z_{k_p}|^2} \wedge \eta,
    \]
    in $U$, where $\eta$ is a smooth $(n-k_p,n-k_p)$. Moreover, after a possible relabeling of the coordinate functions $z_{1},\hdots,z_{k_p}$, locally we have that $|s_{\ell}|_\ell^2 = |z_{\ell}|^2 e^{-\phi_{\ell}}$, where $\phi_{\ell} \in \mathscr{C}^\infty(U)$ is a local weight for the metric $|\cdot|_\ell$ in $U$, for $\ell=1,\hdots,k_p$. Thus, by \Cref{lem:} we can find a smooth $(n-k_p,n-k_p)$-form $\widetilde{\omega}$ such that
    \[
        \omega = \bigg( \bigwedge_{j=1}^{k_p} \partial\log|s_j|_j^2\wedge\bar{\partial}\log|s_j|_j^2 \bigg) \wedge \widetilde{\omega},
    \]
    in some (possibly smaller) neighborhood $\widetilde{U} \ni p$ contained in $U$. Now, covering $X$ by open sets $V$ as above and introducing a partition of unity subordinate to this cover, the result is immediate.
\end{proof}
Now, suppose we are in the setting of \Cref{prop:6}, and consider the products on the right-hand side of \cref{eq:generalized_metric_dependence_formula}. We claim that they indeed can be defined in analogy with \Cref{def:product}. The key observation is the following: The distribution-valued mapping $\lambda \mapsto |s_j|^{2\lambda}_j$, for $j=k+1,\hdots,q$, is holomorphic in a neighborhood of $\lambda = 0$ and has the Taylor expansion
\begin{equation}
    \label{eq:taylorexp}
    |s_j|_j^{2\lambda} = \sum\limits_{\ell=0}^\infty \lambda^\ell \mu_\ell^{|s_j|_j}(1),
\end{equation}
about the origin, where 
\[
    \mu_{\ell}^{|s_j|_j}(1) = \frac{1}{\ell!}\frac{\d^{\ell}}{\d \lambda^{\ell}}\big( |s_j|_j^{2\lambda} \big)\Big|_{\lambda = 0} = \frac{1}{\ell!}\big(\log|s_j|_j^2\big)^{\ell},
\]
c.f. \cref{eq:residue}.

Recall that $J=(J_1,\hdots,J_k)$, where $1\leq J_1 < \cdots < J_k \leq q$, and $k\leq \min\{n,q\}$. Without loss of generality we may assume that $J = \{1,\hdots, k\}$. Let
\begin{equation}
    \label{eq:omegajcases}
    \omega_j = 
    \begin{cases}
        \bar{\partial} \log|s_j|_j^2 \wedge \partial \log|s_j|_j^2 &\text{ if } j \leq k \\
        1 &\text{ if } k+1 \leq 1 \leq q.
    \end{cases}
\end{equation}
Now we define the product
\begin{equation}
    \label{eq:newproduct}
    \mu_{\ell_1}^{|s_1|_q}(\omega_1) \wedge \cdots \wedge \mu_{\ell_q}^{|s_q|_q}(\omega_q)
\end{equation}
precisely as in \Cref{def:product} with respect to the natural regularization of $\omega_j$ for each $j$, that is,
\begin{equation}
    L^j_\ell(\lambda) = 
    \begin{cases}
        \frac{1}{(\ell+1)!} \bar{\partial}\big( ( \log|s|_j^2 )^{\ell+1} |s|_j^{2\lambda} \big) \wedge \partial \log |s|_j^2, \text{ for } \ell \geq -1 &\text{if } j \leq n, \\
        \frac{1}{\ell!}|s_j|_j^{2\lambda} (\log|s_j|_j^2 )^{\ell}, \text{for } \ell \geq 0 &\text{if } k+1 \leq j \leq q, \\
        0 &\text{otherwise}.
    \end{cases}
\end{equation}
To see that \cref{eq:newproduct} is well-defined, commutative and associative we must verify \Cref{prop:3} and \Cref{prop:4} in this more general setting. Since both \Cref{prop:3} and \Cref{prop:4} are local statements, it suffices to fix a point $p \in X$ and consider \cref{eq:newproduct} in a small neighborhood $U \ni p$.

Note that by the assumption that $s_{j_1},\hdots,s_{j_m}$ is a locally complete intersection for each $1\leq j_1<\hdots<j_m \leq q$ ($m \leq \min\{n,q\}$), there is some $I \subseteq \{1,\hdots,q\}$ such that $|I| \leq n$ with $s_j|_U \neq 0$ for each $j \notin I$. Thus, $\mu_{\ell_j}^{|s_j|_j}(\omega_j)$ is smooth in $U$ for each $j \notin I$. Thus, we reduce to the case where $q \leq n$. The proofs of \Cref{prop:3} and \Cref{prop:4} can then be carried through in the same way with the obvious modifications.

Similarly, \Cref{thm:1} holds in the slightly more general setting when $\omega_j$ is of the form \cref{eq:omegajcases}, for $j=1,\hdots,q$. Indeed, \Cref{thm:1} is too a local statement, and can be checked in a neighborhood of a point $p \in X$. Again, the assumption that $s_{j_1},\hdots,s_{j_m}$ is a locally complete intersection for each $1\leq j_1<\hdots<j_m \leq q$ ($m \leq \min\{n,q\}$) means that we can reduce to the case where $q \leq n$. Then the proof can be carried through in the same way with the obvious modifications. Note that, since $\mu_{-1}^{|s_j|_j}(\omega_j) = 0$ for each $k+1\leq j \leq q$, it immediately follows, in view of \cref{eq:mainresult}, that $\mu_{\ell}^{|s|}(\omega_1\wedge \cdots \wedge \omega_q) = 0$ for $\ell < -k$.
\begin{proof}[Proof of \Cref{prop:6}]
    In view of \cref{eq:omegajcases}, $\omega_J = \omega_1 \wedge \cdots \wedge \omega_q$. By \Cref{thm:1} in the generalized setting, 
    \begin{align}
        \label{eq:prop5.4eq1}
        \sum\limits_{\ell=-k}^\infty \lambda^\ell \mu_\ell^{|s|}(\omega_J) &= \sum\limits_{\ell=-k}^\infty \lambda^\ell \mu_\ell^{|s|}(\omega_1 \wedge \cdots \wedge \omega_q) \\
        \nonumber
        &= \bigg( \sum\limits_{\ell_1 = -1}^\infty \lambda^{\ell_1} \mu_{\ell_1}^{|s_1|_1}(\omega_1) \bigg) \wedge \cdots \wedge \bigg( \sum\limits_{\ell_q= 0}^\infty \lambda^{\ell_q} \mu_{\ell_q}^{|s_q|_q}(\omega_q) \bigg).
    \end{align}
    By the standard version of \Cref{thm:1} we also have that
    \begin{equation}
        \label{eq:prop5.4eq2}
        \bigg( \sum\limits_{\ell_1 = -1}^\infty \lambda^{\ell_1} \mu_{\ell_1}^{|s_1|_1}(\omega_1) \bigg) \wedge \cdots \wedge \bigg( \sum\limits_{\ell_k= -1}^\infty \lambda^{\ell_k} \mu_{\ell_k}^{|s_k|_k}(\omega_k) \bigg) = \sum\limits_{j=-k}^\infty \lambda^j \mu_{j}^{|s_J|_J}(\omega_J).
    \end{equation}
    Moreover, in view of \cref{eq:taylorexp} and \cref{eq:omegajcases}, we have that
    \[
        \bigg( \sum\limits_{\ell_{k+1}= 0}^\infty \lambda^{\ell_{k+1}} \mu_{\ell_{k+1}}^{|s_{k+1}|_{k+1}}(\omega_{k+1}) \bigg)\wedge \cdots \wedge \bigg( \sum\limits_{\ell_q= 0}^\infty \lambda^{\ell_q} \mu_{\ell_q}^{|s_q|_q}(\omega_q) \bigg) = \Big(|s_{k+1}|_{k+1}^{2} \cdots |s_q|^2_q\Big)^\lambda \!\!.
    \]
    Thus, by \cref{eq:prop5.4eq1} and \cref{eq:prop5.4eq2} we find that
    \begin{equation}
        \label{eq:prop5.4eq3}
        \sum\limits_{\ell=-k}^\infty \lambda^\ell \mu_\ell^{|s|}(\omega_J) = \Big(|s_{k+1}|_{k+1}^{2} \cdots |s_q|^2_q\Big)^\lambda \sum\limits_{j=-k}^\infty \lambda^j \mu_{j}^{|s_J|_J}(\omega_J).
    \end{equation}
    Expanding $\Big(|s_{k+1}|_{k+1}^{2} \cdots |s_q|^2_q\Big)^\lambda$ in a Taylor series about $\lambda = 0$ we find that
    \begin{align*}
        \Big(|s_{k+1}|_{k+1}^{2} \cdots |s_q|^2_q\Big)^\lambda &= \sum\limits_{\ell=0}^\infty \frac{\lambda^\ell}{\ell!}\Big( \log\big(|s_{k+1}|_{k+1}^2 \cdots |s_q|_q^2\big) \Big)^\ell \\
        &= \sum\limits_{\ell=0}^\infty \frac{\lambda^{\ell}}{\ell!} \bigg( \log\frac{|s|^2}{|s_J|^2_J}\bigg)^{\ell}.
    \end{align*}
    Thus, by \cref{eq:prop5.4eq3}
    \begin{align*}
        \sum\limits_{\ell=-k}^\infty \lambda^\ell \mu_\ell^{|s|}(\omega_J) &= \sum\limits_{\ell'=0}^\infty \frac{\lambda^{\ell'}}{\ell'!} \bigg( \log\frac{|s|^2}{|s_J|^2_J}\bigg)^{\ell'}\sum\limits_{j=-k}^\infty \lambda^j \mu_{j}^{|s_J|_J}(\omega_J) \\
        &= \sum\limits_{\ell'=0}^\infty\sum\limits_{j=-k}^\infty \frac{\lambda^{\ell' + j}}{\ell'!} \bigg( \log\frac{|s|^2}{|s_J|^2_J}\bigg)^{\ell'} \mu_{j}^{|s_J|_J}(\omega_J).
    \end{align*}
    Comparing coefficients in front of $\lambda^\ell$, \cref{eq:generalized_metric_dependence_formula} follows by a straight-forward calculation, completing the proof.
\end{proof}

\section{Examples on $\mathbb{P}^n$}

First, some notation. Let $[Z_0 : \cdots : Z_n]$ be homogeneous coordinates on $\mathbb{P}^n$, regarded as sections of $\mathcal{O}(1)\rightarrow \mathbb{P}^n$. Let $z = (z_1,\hdots,z_n)$ be local holomorphic coordinates on the open chart $\{Z_0 \neq 0 \} \subset \mathbb{P}^n$ defined by $z_j = Z_j/Z_{0}$ for $j=1,\hdots,n$. Let $\omega_{\mathrm{FS}}$ be the Fubini--Study form on $\mathbb{P}^n$, given by
\[
    \omega_{\mathrm{FS}} = \frac{i}{2}\partial\bar{\partial}\log(1+|z|^2),
\]
in $\{Z_0 \neq 0 \}$, where $|z|^2 = |z_1|^2 + \cdots + |z_n|^2$. Let $\|\cdot\|$ be the Fubini--Study metric on $\mathcal{O}(1)$. Then,
\[
    \|Z_j\|^2 = \frac{|Z_j|^2}{|Z|^2},
\]
where $|Z|^2 = |Z_0|^2 + \cdots + |Z_n|^2$. Let
\[
    \omega = \frac{i^n}{2^n}\frac{\d z_1 \wedge \cdots \wedge \d \bar{z}_n}{|z_1\cdots z_n|^2}.
\]
It is straight-forward to check that
\begin{equation}
    \label{eq:omegaformula}
    \omega = \frac{1}{n!}\frac{1}{\|Z_0\|^2\cdots \|Z_n\|^2} \omega_{\mathrm{FS}}^{\wedge n},
\end{equation}
from which it is clear that $\omega$ is a quasi-meromorphic top form on $\mathbb{P}^n$ with singularities along the normal crossings divisor $s = Z_0\cdots Z_n \colon \mathcal{O}(n+1) \rightarrow \mathbb{P}^n$. We will compute
\[
    \mathrm{fp}\int_{\mathbb{P}^n}\omega = \langle \mu_0^{\|s\|}(\omega),1\rangle
\]
explicitly when $n = 2$ and $n = 3$. 

\medskip 

One observation that allows us to compute the following examples explicitly is the following formula.
\begin{conjecture}
    \label{conj:1}
    \begin{equation}
        \label{eq:conj_formula}
        \frac{i^n}{2^n}\frac{\d z_1 \wedge  \cdots \wedge \d \bar{z}_n}{|z_1 \cdots z_n|^2} = \sum\limits_{\ell=0}^n 
        \frac{\ell + 1}{(n-\ell)!} \omega_{\mathrm{FS}}^{\wedge \ell} \wedge \bigg( \sum\limits_{k=0}^n \frac{i}{2}\partial \log \|Z_k\|^2 \wedge \bar{\partial}\log\|Z_k\|^2\bigg)^{\!\!\wedge (n-\ell)}\!.
    \end{equation}
\end{conjecture}
The formula \cref{eq:conj_formula} has been checked to hold for $n\leq 4$, using the algebra software system Macaulay2, and we conjecture that it holds in general. We note that \cref{eq:conj_formula} is an explicit instance of \cref{eq:omegadecomp}. This becomes more clear by the observation
\begin{equation}
    \label{eq:alternateformula}
    \bigg( \sum\limits_{k=0}^n \frac{i}{2}\partial \log \|Z_k\|^2 \wedge \bar{\partial}\log\|Z_k\|^2\bigg)^{\!\!\wedge (n-\ell)} \!\!\!\!\! = \sum\limits_{\mathclap{0\leq j_1 < \cdots < j_{n-\ell} \leq n}} (n-\ell)! \bigwedge\limits_{k=1}^{n-\ell} \frac{i}{2}\partial\log\|Z_{j_k}\|^2 \wedge \bar{\partial}\log\|Z_{j_k}\|^2.
\end{equation}
Thus, by \cref{eq:linearity} and \cref{eq:additivity},
\begin{align*}
    \mu_0^{\|s\|}(\omega) &= \mu_0^{\|s\|} \Bigg( \sum\limits_{\ell=0}^n \frac{\ell + 1}{(n-\ell)!} \omega_{\mathrm{FS}}^{\wedge \ell} \wedge \bigg( \sum\limits_{k=0}^n \frac{i}{2}\partial \log \|Z_k\|^2 \wedge \bar{\partial}\log\|Z_k\|^2\bigg)^{\!\!\wedge (n-\ell)} \Bigg) \\
    &= \sum\limits_{\ell=0}^n(\ell+1) \omega_{\mathrm{FS}}^{\wedge \ell} \wedge \!\!\!\!\!\!\!\!\sum\limits_{0\leq j_1 < \cdots < j_{n-\ell} \leq n} \!\!\!\!\!\!\!\mu_0^{\|s\|} 
    \bigg( \bigwedge\limits_{k=1}^{n-\ell} \frac{i}{2}\partial\log\|Z_{j_k}\|^2 \wedge \bar{\partial}\log\|Z_{j_k}\|^2 \bigg).
\end{align*}
By applying \Cref{prop:6} and \Cref{thm:1} we obtain
\begin{equation}
    \label{eq:masterformulaPn}
    \begin{aligned}
        \mu_0^{\|s\|}(\omega) &= \sum\limits_{\ell=0}^n (\ell+1) \omega_{\mathrm{FS}}^{\wedge \ell} \wedge \!\!\!\!\!\! \sum\limits_{0\leq j_1 < \cdots < j_{n-\ell} \leq n}\sum\limits_{\ell' = 0}^{n-\ell} \frac{1}{\ell'}\bigg( \log\frac{\|Z_0 \cdots Z_{n}\|^2}{\|Z_{j_1}\cdots Z_{j_{n-\ell}}\|^2}\bigg)^{\ell'} \times \\
        &\qquad\qquad \times \sum\limits_{\substack{\ell_1,\hdots,\ell_{n-\ell} \geq -1 \\ \ell_1 + \cdots + \ell_{n-\ell} = - \ell'}} \bigwedge_{k=1}^{n-\ell} \mu_{\ell_k}^{\|Z_{j_k}\|}\Big(\frac{i}{2}\partial\log\|Z_{j_k}\|^2 \wedge \bar{\partial}\log\|Z_{j_k}\|^2\Big),
    \end{aligned}
\end{equation}
which is an explicit instance of \cref{eq:masterformula}. Moreover, by \cref{eq:elementarymuformula}, 
\begin{align*}
    \mu_{\ell_k}^{\|Z_{j_k}\|}\Big(\frac{i}{2}\partial\log\|Z_{j_k}\|^2 \wedge \bar{\partial}\log\|Z_{j_k}\|^2 \Big) &= -\frac{1}{(\ell_k+2)!} \frac{i}{2}\bar{\partial}\partial \big( \log\|Z_{j_k}\|^2\big)^{\ell_k+2} \\
    &\qquad\quad + \frac{\pi}{(\ell_k+1)!}\big(\log\|Z_{j_k}\|^2\big)^{\ell_k+1} c_{1}(\mathcal{O}_1,\|\cdot\|),
\end{align*}
where $c_1(\mathcal{O}(1),\|\cdot\|) = \omega_{\mathrm{FS}}/\pi$.
\begin{example}[$\mathbb{P}^2$]
    For $n=2$,
    \[
        \omega = \frac{i^2}{2^2}\frac{\d z_1 \wedge \d \bar{z}_1 \wedge \d z_2 \wedge \d \bar{z}_2}{|z_1 z_2|^2},
    \]
    and, by \cref{eq:conj_formula},
    \begin{align*}
        \omega &= \sum\limits_{0\leq i < j \leq 2} \frac{i}{2}\partial\log\|Z_i\|^2 \wedge \bar{\partial}\log\|Z_i\|^2 \wedge \frac{i}{2}\partial\log\|Z_j\|^2 \wedge \bar{\partial}\log\|Z_j\|^2 \\
        &\qquad + 2 \sum_{i=0}^2 \frac{i}{2} \partial\log\|Z_i\|^2 \wedge \bar{\partial}\log\|Z_i\|^2\wedge \omega_{\text{FS}} + 3 \omega_{\text{FS}}\wedge \omega_{\text{FS}}.
    \end{align*}
    Let $\omega_j = \tfrac{i}{2}\partial\log\|Z_j\|^2 \wedge \bar{\partial}\log\|Z_j\|^2$. By \cref{eq:masterformulaPn},
    \begin{equation}
        \label{eq:p2formula}
        \begin{aligned}
            \mu_0^{\|s\|}(\omega) &= \sum\limits_{0\leq i < j \leq 2} \sum\limits_{\ell=0}^{2}\frac{1}{\ell!}\bigg(\log\frac{\|Z_0 Z_1 Z_2\|^2}{\|Z_i Z_j\|^2}\bigg)^{\ell}\sum\limits_{\substack{\ell_1,\ell_2 \geq -1\\\ell_1 + \ell_2 = - \ell}}\mu_{\ell_1}^{\|Z_i\|}(\omega_i) \wedge \mu_{\ell_2}^{\|Z_j\|} (\omega_j) \\
            &\qquad + 2 \sum\limits_{j=0}^2 \omega_{\mathrm{FS}} \wedge \sum\limits_{\ell=0}^1 \frac{1}{\ell!} \bigg( \log\frac{\|Z_0 Z_1 Z_2\|^2}{\|Z_j\|^2} \bigg)^{\ell} \mu_{-\ell}^{\|Z_j\|}(\omega_j) + 3 \omega_{\mathrm{FS}}\wedge \omega_{\mathrm{FS}}.
        \end{aligned}
    \end{equation}
    When computing $\langle\mu_0^{\|s\|}(\omega),1\rangle$ some simplifications can be made. First, the resulting integrals can be computed in any chart $\{Z_j\neq 0\} \simeq \C^2$. Furthermore, we have that
    \[
        \Big\langle \big(\log\|Z_0\|^2\big)^{\ell} \mu_{\ell_1}^{\|Z_1\|}(\omega_1) \wedge \mu_{\ell_2}^{\|Z_2\|} (\omega_2), 1\Big\rangle
    \]
    is independent under permutations of $(Z_0,Z_1,Z_2)$, and the same goes for the other terms on the right-hand side of \cref{eq:p2formula}. Thus, we obtain a simplified formula for the finite part
    \begin{align*}
        \langle\mu_0^{\|s\|}(\omega),1\rangle &= 3 \sum\limits_{\ell=0}^{2} \sum\limits_{\substack{\ell_1,\ell_2 \geq -1\\\ell_1 + \ell_2 = - \ell}} \frac{1}{\ell!} \Big\langle \big(\log\|Z_0\|^2\big)^{\ell} \mu_{\ell_1}^{\|Z_1\|}(\omega_1) \wedge \mu_{\ell_2}^{\|Z_2\|} (\omega_2), 1\Big\rangle\\
        &\qquad + 6 \sum\limits_{\ell=0}^1 \Big\langle \omega_{\mathrm{FS}}\wedge \big( \log\|Z_0 Z_1\|^2 \big)^{\ell} \mu_{-\ell}^{\|Z_2\|}(\omega_2), 1\Big\rangle + 3 \Big\langle \omega_{\mathrm{FS}}\wedge \omega_{\mathrm{FS}}, 1\Big\rangle \\
        &= 3 \Big\langle \mu_{0}^{\|Z_1\|}(\omega_1)\wedge \mu_{0}^{\|Z_2\|}(\omega_2),1 \Big\rangle + 6 \Big\langle \mu_{-1}^{\|Z_1\|}(\omega_1)\wedge \mu_{1}^{\|Z_2\|}(\omega_2),1\Big\rangle \\
        &\qquad+ 6\Big\langle \log\|Z_0\|^2  \mu_{0}^{\|Z_1\|}(\omega_1)\wedge \mu_{-1}^{\|Z_2\|}(\omega_2),1 \Big\rangle \\
        &\qquad + \frac{3}{2} \Big\langle \big( \log\|Z_0\|^2 \big)^2 \mu_{-1}^{\|Z_1\|}(\omega_1)\wedge \mu_{-1}^{\|Z_2\|}(\omega_2),1\Big\rangle \\
        &\qquad + 6 \Big\langle \omega_{\mathrm{FS}}\wedge \mu_0^{\|Z_2\|}(\omega_2),1\Big\rangle +6\Big\langle \omega_{\mathrm{FS}}\wedge \log\|Z_0 Z_1\|^2 \mu_{-1}^{\|Z_2\|}(\omega_2),1\Big\rangle \\
        &\qquad +  3\Big\langle \omega_{\mathrm{FS}}\wedge \omega_{\mathrm{FS}}, 1\Big\rangle.
    \end{align*}
    Note that after using \cref{eq:elementarymuformula} to expand the $\mu_\ell^{\|Z_j\|}(\omega_j)$, some resulting terms in the expression for $\mu_0^{\|s\|}(\omega)$ are exact and hence vanish acting on $1$. This reduces the number of integrals one ultimately needs to compute. Computing the individual contributions to $\langle\mu_0^{\|s\|}(\omega),1\rangle$ is straight-forward, albeit tedious, to do by hand. As an example, consider the following term,
    \begin{align*}
        \Big\langle \log\|Z_0\|^2 \mu_0^{\|Z_1\|}(\omega_1) \wedge \mu_{-1}^{\|Z_2\|}(\omega_2),1\Big\rangle &= -\frac{\pi i}{4}\int_{\mathbb{P}^2} \log\|Z_0\|^2 \bar{\partial}\partial\big(\log\|Z_1\|^2\big)^2 \wedge [Z_2 = 0] \\
        &\qquad + \pi \int_{\mathbb{P}^2}\log\|Z_0\|^2\log\|Z_1\|^2 \omega_{\mathrm{FS}}\wedge [Z_2 = 0] \\
        &= \frac{\pi i}{4}\int_{\mathbb{P}^2} \bar{\partial}\log\|Z_0\|^2\wedge \partial\big(\log\|Z_1\|^2\big)^2 \wedge [Z_2 = 0] \\
        &\qquad + \pi \int_{\mathbb{P}^2}\log\|Z_0\|^2\log\|Z_1\|^2 \omega_{\mathrm{FS}}\wedge [Z_2 = 0] \\
        &= \frac{\pi i}{4}\int_{\mathbb{P}^1} \bar{\partial}\log\|Z_0\|^2\wedge \partial\big(\log\|Z_1\|^2\big)^2 \\
        &\qquad + \pi \int_{\mathbb{P}^1}\log\|Z_0\|^2\log\|Z_1\|^2 \omega_{\mathrm{FS}}.
    \end{align*}
    With local coordinate $z = Z_1/Z_0$ in the chart $\{Z_0\neq 0\}\cap \{Z_2=0\} \subset \{Z_2 = 0\} \simeq \mathbb{P}^1$, we thus have that
    \begin{align*}
        \Big\langle \log\|Z_0\|^2 \mu_0^{\|Z_1\|}(\omega_1) \wedge \mu_{-1}^{\|Z_2\|}(\omega_2),1\Big\rangle &=\frac{\pi i}{4}\int_{\C} \bar{\partial}\log\frac{1}{1+|z|^2}\wedge \partial\bigg(\log\frac{|z|^2}{1+|z|^2}\bigg)^2\\
        &\quad + \frac{\pi i}{2} \int_{\mathbb{C}}\log\frac{1}{1+|z|^2}\log\frac{|z|^2}{1+|z|^2} 
        \frac{\d z\wedge \d \bar{z}}{{(1+|z|^2)}^2} \\
        &= \frac{\pi i}{2}\int_{\C} \log\frac{|z|^2}{1+|z|^2} \frac{\d z \wedge \d \bar{z}}{{(1+|z|^2)}^2}\\
        &\quad + \frac{\pi i}{2} \int_{\mathbb{C}}\log\frac{1}{1+|z|^2}\log\frac{|z|^2}{1+|z|^2} 
        \frac{\d z\wedge \d \bar{z}}{{(1+|z|^2)}^2}.
    \end{align*}
    Now, changing to polar coordinates $(r,\theta)$, we obtain
    \begin{align*}
        \Big\langle \log\|Z_0\|^2 \mu_0^{\|Z_1\|}(\omega_1) \wedge \mu_{-1}^{\|Z_2\|}(\omega_2),1\Big\rangle &= 2\pi^2 \int_{0}^\infty  \log\frac{r^2}{1+r^2} \frac{r \,\d r}{{(1+r^2)}^2}\\
        &\qquad +2\pi^2 \int_{0}^\infty\log\frac{1}{1+r^2}\log\frac{r^2}{1+r^2} 
        \frac{r \,\d r}{{(1+r^2)}^2}.
    \end{align*}
    Finally, we change to a coordinate on the unit interval, via $x = r^2/(1+r^2)$, with $\d x = 2 r \,\d r/(1+r^2)^2$, and obtain
    \begin{align*}
        \Big\langle \log\|Z_0\|^2 \mu_0^{\|Z_1\|}(\omega_1) \wedge \mu_{-1}^{\|Z_2\|}(\omega_2),1\Big\rangle &=  \pi^2 \int_{0}^1  \log x \big( 1 + \log(1-x)\big)\,\d x \\
        &= \pi^2 \big(1 -\zeta(2) \big).
    \end{align*}

    \medskip
    
    After adding up each contribution, we end up with the following finite part,
    \begin{equation}
        \label{eq:p2}
        \mathrm{fp} \,\big(\tfrac{i}{2}\big)^2 \int_{\mathbb{P}^2} \frac{\d z_1 \wedge \d \bar{z}_1 \wedge \d z_2 \wedge \d \bar{z}_2}{|z_1 z_2|^2} = - 9\pi^2 \zeta(2).
    \end{equation}
\end{example}
\begin{example}[$\mathbb{P}^3$]
    The $n=3$ case is worked out analogously to $n=2$ but there are many more integrals to compute. We ultimately obtain the following finite part,
    \begin{equation}
        \label{eq:p3}
        \mathrm{fp}\,\big(\tfrac{i}{2}\big)^3\int_{\mathbb{P}^3} \frac{\d z_1 \wedge \d \bar{z}_1 \wedge \d z_2 \wedge \d \bar{z}_2  \wedge \d z_3 \wedge \d \bar{z}_3}{|z_1 z_2 z_3|^2} = 80 \pi^3 \zeta(3).
    \end{equation}
\end{example}

\subsection{Alternative approach}

There turns out to be another way to compute the finite part $\mathrm{fp}\int_{\mathbb{P}^n} \omega$ in this particular example, for a general $n$. The following is based on a similar computation found in the proof of \cite[Proposition 6.3]{Ber1}. The key is to consider the integral
\[
    I(\lambda) = \int_{\mathbb{C}^{n+1}} |Z_0 \cdots Z_{n}|^{2(\lambda-1)} e^{-\sum_{j=0}^n |Z_j|^2} \d m,
\]
for $\mathfrak{Re}\,\lambda \gg 0$, where $\d m$ is the standard Lebesgue measure on $\mathbb{C}^{n+1}$, that is $\d m = \frac{i^{n+1}}{2^{n+1}} \d Z_0 \wedge \d \bar{Z}_0 \wedge \cdots \wedge \d Z_n \wedge \d \bar{Z}_n$. Changing to spherical coordinates, and using the homogeneity of $Z_0\cdots Z_n$, we find that
\[
    I(\lambda) = \int_{\mathbb{S}^{2n+1}} \bigg(\frac{|Z_0 \cdots Z_{n}|^2}{{\big(\sum_{j=0}^n|Z_j|^2\big)}^{n+1}}\bigg)^{\lambda-1}\d \sigma \int_0^{\infty} r^{2(\lambda-1)(n+1)} e^{- r^2} r^{2(n+1)-1} \d r,
\]
where we note that $|Z_0 \cdots Z_n|^2/\big(\sum_{j=0}^n|Z_j|^2\big)^{n+1}$ is independent of the radius $r$, and thus defines a function on the unit $(2n+1)$-sphere $\mathbb{S}^{2n+1}$. One then notes that for any $k$-homogeneous polynomial $p=p(Z_0, \hdots, Z_n)$ on $\mathbb{P}^n$, in particular $p=Z_0\cdots Z_n$, we have that
\begin{equation}
    \int_{\mathbb{S}^{2n+1}}|p(Z_0/|Z|,\hdots,Z_n/|Z|)|^{2\lambda} \d \sigma = \frac{2 \pi}{n!} \int_{\mathbb{P}^n}\|p\|^{2\lambda}\omega_{\mathrm{FS}}^n,
\end{equation}
where $\d \sigma$ is the uniform volume form on the $(2n+1)$-sphere and where $\|\cdot\|$ is the Fubini--Study metric. Thus, we find that
\begin{align*}
    I(\lambda) &= \frac{2\pi}{n!} \int_{\mathbb{P}^n}\|s\|^{2(\lambda-1)} \omega_{\mathrm{FS}}^n \int_0^{\infty} r^{2(\lambda-1)(n+1)} e^{- r^2} r^{2(n+1)-1} \d r \\
    &= 2\pi \int_{\mathbb{P}^n}\|s\|^{2\lambda} \omega \int_0^{\infty} r^{2(\lambda-1)(n+1)} e^{- r^2} r^{2(n+1)-1} \d r,
\end{align*}
in view of \cref{eq:omegaformula}, that is,
\[
    \int_{\mathbb{P}^n}\|s\|^{2\lambda} \omega = \frac{1}{2\pi} I(\lambda) \bigg(\int_0^{\infty} r^{2(\lambda-1)(n+1)} e^{- r^2} r^{2(n+1)-1} \d r\bigg)^{-1}.
\]
The last two observations are that
\begin{align*}
    I(\lambda) &= \prod_{j=0}^n \frac{i}{2}\int_{\mathbb{C}}|Z_j|^{2(\lambda-1)} e^{-|Z_j|^2} \d Z_j \wedge \d \bar{Z}_j \\
    &= (2\pi)^{n+1} \bigg( \int_0^\infty r^{2(\lambda-1)+1} e^{-r^2} \d r\bigg)^{n+1} \\
    &= \pi^{n+1} \Gamma(\lambda)^{n+1},
\end{align*}
and
\begin{align*}
    \int_0^{\infty} r^{2(\lambda-1)(n+1)} e^{- r^2} r^{2(n+1)-1} \d r &= \frac{1}{2}\Gamma((n+1)\lambda),
\end{align*}
where $\Gamma$ denotes the Euler gamma function. Thus
\[
    \int_{\mathbb{P}^n} \|s\|^{2\lambda} \omega = \pi^n \frac{\Gamma(\lambda)^{n+1}}{\Gamma((n+1)\lambda)},
\]
where the right hand side defines a meromorphic function on $\C$, with a pole of order $n$ at $0$. Then, in view of \cref{eq:residue},
\begin{equation}
    \label{eq:adhocformula}
    \mathrm{fp}\int_{\mathbb{P}^n}\omega = \frac{1}{n!}\frac{\d^n}{\d \lambda^{n}} \bigg( \lambda^n  \int_{\mathbb{P}^n} \|s\|^{2\lambda} \omega \bigg)\bigg|_{\lambda=0} = \frac{\pi^n}{n!} \frac{\d^n}{\d \lambda^{n}} \bigg( \frac{\lambda^n \Gamma(\lambda)^{n+1}}{\Gamma((n+1)\lambda)} \bigg)\bigg|_{\lambda = 0}.
\end{equation}
Indeed, this formula produces the same answer for $n=2,3$. Moreover, using Wolfram Mathematica we are able to compute the finite part also in the case $n=4$ and $n=5$ using this formula. We obtain the following:
\begin{align}
    \label{eq:p4}
    \mathrm{fp}\int_{\mathbb{P}^4}\omega &= -150\pi^4 \zeta(4), \\
    \label{eq:p5}
    \mathrm{fp}\int_{\mathbb{P}^5}\omega &= 252\pi^5 \big( 37\zeta(5) - 25\zeta(2)\zeta(3)\big).
\end{align}

\end{document}